\documentclass[10pt,a4paper]{article}
\usepackage{latexsym}
\usepackage{amsfonts}
\usepackage{amsmath}
\usepackage{amsthm}
\usepackage{amssymb}
\usepackage{hyperref}
\usepackage{graphicx}
\usepackage{subfigure}
\usepackage{balance}
\usepackage{multicol}
\usepackage{float}
\usepackage{yhmath}
\usepackage{epstopdf}
\usepackage{color}
\usepackage{etoolbox}
\AtBeginEnvironment{cases}{\linespread{1.2}\selectfont}

\def\p{\partial}
\def\ve{\varepsilon}
\def\f{\frac}

\def\la{\lambda}

\def\vp{\varphi}
\def\O{\Omega}
\def\o{\omega}
\def\th{\theta}

\def\G{\Gamma}

\def\vs{\varsigma}
\def\dl{\delta}

\def\ds{\displaystyle}

\def\i{\infty}

\def\no{\nonumber}

\def\beq{\begin{equation}}
\def\eeq{\end{equation}}
\def\ben{\begin{eqnarray}}
\def\een{\end{eqnarray}}
\def\bec{\begin{cases}}
\def\eec{\end{cases}}

\newcommand{\bR}{{\mathbb R}}

\newcommand{\vG}{{\mathcal G}}

\begin{document}
\newtheorem{theorem}{Theorem}
\newtheorem{lemma}{Lemma}
\renewcommand{\thelemma}
{\arabic{section}.\arabic{lemma}}
\newtheorem{corollary}[lemma]{Corollary}
\newtheorem{remark}{Remark}
\renewcommand{\theremark}
{\arabic{section}.\arabic{remark}}
\renewcommand{\theequation}
{\arabic{section}.\arabic{equation}}
\renewcommand{\thetheorem}
{\arabic{section}.\arabic{theorem}}
\makeatletter
\@addtoreset{equation}{section} \makeatother \makeatletter
\makeatletter
\@addtoreset{lemma}{section} \makeatother \makeatletter
\makeatletter
\@addtoreset{remark}{section} \makeatother \makeatletter
\makeatletter \@addtoreset{theorem}{section} \makeatother
\makeatletter
\makeatother
\title{\bf {The shock formation and optimal regularities of the resulting shock curves for 1-D scalar conservation laws
}}
\author{Yin Huicheng$^{1,*}$, \quad Zhu Lu$^{2,}$
\footnote{Yin Huicheng (huicheng@nju.edu.cn, 05407@njnu.edu.cn)
and Zhu Lu (zhulu@hhu.edu.cn) are
supported by the NSFC (No.11571177, No.11731007, No.12001162).}\vspace{0.5cm}\\
1.  School of Mathematical Sciences and Institute of Mathematical Sciences,\\
Nanjing Normal University, Nanjing, 210023, China.
\\
\vspace{0.5cm}
2. College of Science, Hohai University, Nanjing, 210098, China.}
\date{}
\maketitle

\begin{abstract}
The study on the shock formation and the regularities of the resulting shock surfaces
for hyperbolic conservation laws is a basic problem in the nonlinear partial differential equations.
In this paper, we are concerned with the shock formation and the optimal regularities of the resulting shock curves
for the 1-D conservation law $\p_tu+\p_xf(u)=0$ with the smooth initial data $u(0,x)=u_0(x)$.
If $u_0(x)\in C^{1}(\Bbb R)$ and $f(u)\in C^2(\Bbb R)$,
it is well-known that the solution $u$ will blow up on the time $T^*=-\f{1}{\min{g'(x)}}$
when $\min{g'(x)}<0$ holds for $g(x)=f'(u_0(x))$. Let $x_0$ be a local minimum point of $g'(x)$ such that $g'(x_0)=\min{g'(x)}<0$ and
$g''(x_0)=0$, $g^{(3)}(x_0)>0$ (which is called the generic nondegenerate condition), then by Theorem 2 of \cite{Le94},
a weak entropy  solution $u$ together with the shock curve $x=\vp(t)\in C^2[T^*, T^*+\ve)$ starting from
the blowup point $(T^*, x^*=x_0+g(x_0)T^*)$ can be locally constructed. When the generic nondegenerate condition is violated,
namely, when $x_0$ is a local minimum point of $g'(x)$  such that $g''(x_0)=g^{(3)}(x_0)=...=g^{(2k_0)}(x_0)=0$ but $g^{(2k_0+1)}(x_0)>0$
for some $k_0\in\Bbb N$ with $k_0\ge 2$; or $g^{(k)}(x_0)=0$ for any $k\in\Bbb N$ and $k\ge 2$, we will
study the shock formation and the optimal regularity of the shock curve $x=\vp(t)$, meanwhile,
some precise descriptions on the behaviors of $u$ near the  blowup point $(T^*, x^*)$
are given. Our main aims are to show that: around the blowup point, the shock really appears whether the initial data are
degenerate with finite orders or with infinite orders;  the optimal regularities
of the shock solution and the resulting shock curve  have the explicit relations with the degenerate degrees of the initial data.

\end{abstract}

\begin{quote} {\bf Keywords:} {Shock formation}, shock curve, entropy condition, Rankine-Hugoniot condition, hyperbolic conservation law
\end{quote}
\vskip 0.2 true cm

\begin{quote} {\bf Mathematical Subject Classification 2000:} 35L05, 35L72 \end{quote}

\section{Main result}

The study on the blowup and shock formation of smooth solutions to the
hyperbolic conservation laws is a basic problem in the nonlinear partial differential equations,
which has made much progress for the multi-dimensional cases in recent years (see \cite{B-1}-\cite{B-2},
\cite{C2}-\cite{0-Speck}, \cite{LS}-\cite{S2}). In the present paper, we are concerned with the
shock formation and the optimal regularities of the resulting shock curves
for the 1-D conservation law
\begin{equation}\label{0-1}
\left\{
\begin{aligned}
&\p_tu+\p_xf(u)=0,\ (t,x)\in\bR_+\times\bR,\\
&u(0,x)=u_0(x),\ x\in\bR,
\end{aligned}
\right.
\end{equation}
where $f(u)\in C^2(\Bbb R)$ and $u_0(x)\in C^1(\Bbb R)$.
It is well-known that the $C^1$ solution $u$ of \eqref{0-1} will blow up at the time $T^*=-\f{1}{\min{g'(x)}}$
with $g(x)=f'(u_0(x))$ and $\min_{x\in\Bbb R}{g'(x)}<0$. If we further assume $g(x)\in L^{\infty}(\Bbb R)\cap C^p(\Bbb R)$
with $p\ge 4$, and let $x_0$ be a local minimum point of $g'(x)$ such that
\begin{align}\label{0-2}
g'(x_0)=\min_{x\in\bR}{g'(x)}<0,\quad g''(x_0)=0, \quad g^{(3)}(x_0)>0,
\end{align}
which is called the generic nondegenerate condition in \cite{A1}, then by Theorem 2 of \cite{Le94},
a weak entropy  solution $u$  of \eqref{0-1} together with the shock curve $x=\vp(t)$ starting from
the blowup point $(T^*, x^*=x_0+g(x_0)T^*)$ can be locally obtained as follows:

(i)\begin{align}\label{0-3}
&\vp(t)\in C^p(T^*, T^*+\ve)\cap C^{\f{p}{2}}[T^*, T^*+\ve).
\end{align}

(ii) In some part of the neighbourhood of $(T^*,x^*)$ near $x=\vp(t)$, for $t\ge T^*$ and $x\not=\vp(t)$,
\begin{equation}\label{0-4}
\left\{
\begin{aligned}
&|u(t,x)-u(T^*,x^*)|\le C((t-T^*)^3+(x-x^*)^2)^{\f16},\\
&|\p_tu(t,x)|\le\f{C}{((t-T^*)^3+(x-x^*)^2)^{\f16}},\\
&|\p_xu(t,x)|\le\f{C}{((t-T^*)^3+(x-x^*)^2)^{\f13}},\\
&|\p_x^2u(t,x)|\le\f{C}{((t-T^*)^3+(x-x^*)^2)^{\f56}}.
\end{aligned}
\right.
\end{equation}

When the generic nondegenerate condition \eqref{0-2} is violated,
namely, if $x_0$ is a local minimum point of $g'(x)$ such that
\begin{equation}\label{0-5}
\left\{
\begin{aligned}
&g(x)\in L^{\infty}(\Bbb R)\cap C^{2k+2}(\Bbb R)\quad\text{for
$k\in\Bbb N$ with $k\ge 2$,}\\
&g'(x_0)=\min_{x\in\bR}{g'(x)}<0,\quad g''(x_0)=g^{(3)}(x_0)=...=g^{(2k)}(x_0)=0,\quad g^{(2k+1)}(x_0)>0,
\end{aligned}
\right.
\end{equation}
or
\begin{equation}\label{0-6}
\left\{
\begin{aligned}
&g(x)\in L^\infty(\Bbb R)\cap C^{\infty}(\Bbb R),\\
&g'(x_0)=\min_{x\in\bR}{g'(x)}<0,\quad g^{(k)}(x_0)=0\quad \text{for any $k\in\Bbb N$ and $k\ge 2$},
\end{aligned}
\right.
\end{equation}
we will
study the shock formation and the optimal regularity of the resulting shock $x=\vp(t)$ from the blowup point $(T^*, x^*)$,
meanwhile, some precise descriptions on the behaviors of  the solution  $u$ around the blowup point $(T^*, x^*)$
(rather than only some part near the shock curve) will be given.

Without loss of generality and for convenience, we set $x_0=0$ in \eqref{0-5} and \eqref{0-6}. In addition, under condition \eqref{0-5}, near $x_0=0$ we assume
\begin{align}\label{0-7}
&g(x)=-x+x^{2k+1}+r(x),
\end{align}
where $r(x)\in C^{2k+2}$ satisfies that $r^{(j)}(x)=O(x^{2k-j+2})$ for $0\le j\le 2k+2$;
under condition \eqref{0-6}, we choose a class of initial data
\begin{align}\label{0-8}
{g(x)=-x+e^{-|x|^{-p}}\left(\f{x}{p}+r_0(x)\right),}
\end{align}
where $p>0$ and $r_0(x)\in C^\i\cap L^\i$ with
\beq\label{0-9}
{r^{(j)}_0(x)=\left\{
                           \begin{array}{ll}
                             O(x^{2-j}), & j=0,1,2, \\
                             O(1), & j\ge 3
                           \end{array}
                         \right.
}
\eeq
for $x$ near $0$.

Starting from the blowup point $(1,0)$ of \eqref{0-1}, let the formed shock curve  $\G$ be denoted by $x=\vp(t)$
if the shock really appears.
On the left hand side and right hand side of $\G$ for $t\ge 1$, the weak entropy solution  $u$ is represented
by $u_-$ and $u_+$ respectively (see Figure 1 below).
\begin{figure}[h]
\centering
\includegraphics[scale=0.8]{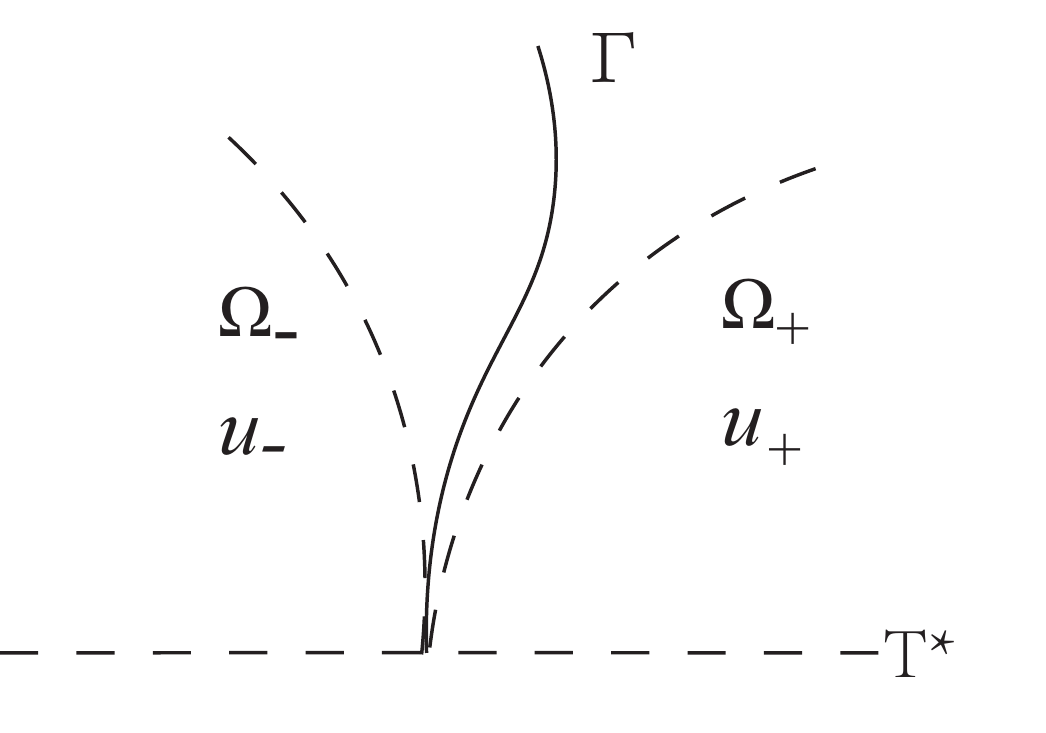}
\caption{\bf Shock formation}
\end{figure}

It follows from the Rankine-Hugoniot condition and entropy condition
on $\G$ that
\begin{align}\label{Rankine-Hugoniot Condition}
&\vp'(t)[u](t,\vp(t))=[f(u)](t,\vp(t)),
\end{align}
where $[u](t,\vp(t))=u_+(t,\vp(t))-u_-(t,\vp(t))$ is the jump of $u$ across $\G$, and

\begin{align}\label{entropy condition}
f'(u_+(t,\vp(t)))<\vp'(t)<f'(u_-(t,\vp(t))).
\end{align}
Our main results are
\begin{theorem}\label{main result1}
Under assumption \eqref{0-7}, there exists a unique solution  $u\in C^1((0,1)\times\bR)\cap C([0,1]\times\bR)$ to
problem \eqref{0-1} together with \eqref{Rankine-Hugoniot Condition}-\eqref{entropy condition} for $t\ge 1$.
Furthermore,

(1) $\vp(t)\in C^\f{k+1}{k}[1,1+\ve)$ and
$u\in C^1((1,1+\ve)\times\bR)\setminus \{x=\vp(t)\})$ for some $\ve>0$.

(2) near the blowup point $(1,0)$, the behaviors of $u$ and its derivatives are as follows
\ben
\label{1.1}|u(t,x)-u(1,0)|&=&O(|t-1|^\f{1}{2k}+|x|^{\f{1}{2k+1}}),\\
\label{1.2}|\p_t u(t,x)|&=&O((|t-1|^{\f{1}{2k}}+|x|^{\f{1}{2k+1}})^{-(2k-1)}),\\
\label{1.3}|\p_x u(t,x)|&=&O((|t-1|^{\f{1}{2k}}+|x|^{\f{1}{2k+1}})^{-2k}).
\een
\end{theorem}

\begin{theorem}\label{main result2}
Under assumption \eqref{0-8}, there exists a unique solution  $u\in C^1((0,1)\times\bR)\cap C([0,1]\times\bR)$ to
problem \eqref{0-1} together with \eqref{Rankine-Hugoniot Condition}-\eqref{entropy condition} for $t\ge 1$.
Furthermore,

(i) $\vp(t)\in C^1[1,1+\ve]$ and
$u\in C^1(((1,1+\ve)\times\bR)\setminus \{x=\vp(t)\})$ for some $\ve>0$. In addition, $\vp(t)=O((t-1)|\ln(t-1)|^{-\f{2}{p}})$
near $t=1$ and for $t>1$.

(ii) near the blowup point $(1,0)$, the behaviors of $u$ and its derivatives are as follows
\ben
\label{2.1}|u(t,x)-u(1,0)|&=&O(|\ln|t-1||^{-\f{1}{p}}+|\ln |x||^{-\f{1}{p}}),\\
\label{2.2}|\p_t u(t,x)|&=&O(|t-1|^{-1}|\ln|t-1||^{-1-\f{1}{p}}+|x|^{-1}|\ln|x||^{-1-\f{1}{p}}),\\
\label{2.3}|\p_x u(t,x)|&=&O(|t-1|^{-1}|\ln|t-1||^{-1}+|x|^{-1}|\ln|x||^{-1}).
\een
\end{theorem}

\begin{remark}
If we take $k=1$, then Theorem 1.1 coincides with the result in Theorem 2 of \cite{Le94}.
In addition, the author in  \cite{Le94} only shows the behaviors of $u$ in some part of the
neighbourhood of the blowup point  $(1,0)$, which corresponds to
the smallness of the variable $|\la|=\f{|x|}{(t-1)^{\f32}}$ for $t>1$. This can be referred to
the proof of Lemma 2.1 in \cite{Le94}, where $|\la|$ is required to be small.
\end{remark}

\begin{remark}
The regularities of $\vp(t)$ in Theorem 1.1 and  Theorem 1.2 are optimal.
One can see Remark 3.1 and Remark 4.1 below.
\end{remark}

\begin{remark}
Under the generic nondegenerate assumption of the initial data, for the 1-D $2\times 2$ $p-$ system of
polytropic gases, the authors in \cite{Kong}-\cite{Le94} and \cite{Chen-Dong} obtain the formation and  construction
of the shock wave starting from the blowup point under some variant conditions; for the 1-D $3\times 3$
strictly hyperbolic conservation laws with the small initial data or the 3-D full compressible Euler equations with
symmetric structure and small perturbation,
the authors in \cite{Chen-Xin-Yin}, \cite{Yin1} and \cite{C-L} also get the formation and  construction
of the resulting shock waves, respectively.
\end{remark}

In order to prove Theorem 1.1-1.2,
our focus is to solve the singular and nonlinear ordinary differential equation \eqref{Rankine-Hugoniot Condition}
as in \cite{Le94}. Note that the equation \eqref{Rankine-Hugoniot Condition} is equivalent to
$\vp'(t)=G(t,\vp(t))\triangleq \int_0^1f'(\th u_+(t,\vp(t))+(1-\th)u_-(t,\vp(t)))d\th$, where
the function $G(t,\vp)$ is not Lipschtzian with respect to variable $\vp$ since the first order derivative of
$u_{\pm}(t,x)$ with respect to the variable $x$ admits the strong singularities (see (1.13) and (1.17)). To get the uniqueness and
regularity of $(\vp(t), u_{\pm}(t,x))$, we require to carefully analyze the behavior and regularity
of solution $u$ near the blowup point $(1,0)$. Due to the more degenerate conditions \eqref{0-5}
and \eqref{0-6}, we shall introduce some different transformations of $(t,x)$ from that in \cite{Le94}
(for examples, see \eqref{3-4},  \eqref{3-15}, \eqref{4-4} and so on).
By involved computation, the behaviors of solution $u$ around the point $(1,0)$ are
derived and then the optimal regularities of $\vp(t)$ are also established. From our results,
we have known two basic facts for problem \eqref{0-1}: (1) Around the blowup point, the shock really appears whether the initial data are
degenerate with finite orders or with infinite orders. (2)  The optimal regularities
of the shock solution and the resulting shock curve  have explicit relations with the degenerate degrees of the initial data.

Our paper is organized as follows. In Section 2, we give some basic analysis on the characteristics envelope
of equation \eqref{0-1} near $(1,0)$, meanwhile, the detailed behaviors of the characteristics near  $(1,0)$
are established. The proofs of Theorem 1.1 and Theorem 1.2 are given in Section 3 and Section 4 respectively.

\section{Some preliminary}

For problem \eqref{0-1}, we define the characteristics $x=x(t,y)$ starting from the point
$(0,y)$ as follows
\beq\label{2-1}
\bec
&\ds\f{d x(t,y)}{d t}=f'(u(t,x(t,y))),\\
&x(0,y)=y.
\eec
\eeq
Then along this characteristics we have
\beq\label{2-2}
u(t,x(t,y))\equiv u_0(y).
\eeq
This means that the characteristics $x(t,y)$ is straight and
\beq\label{2-3}
x(t,y)=y+tg(y).
\eeq
For any fixed $t>0$, in order to solve $y=y(t,x)$ in \eqref{2-3} such that the solution $u$ in \eqref{2-2}
can be obtained, it is necessary to let
\beq\label{2-4}
\f{\p x}{\p y}(t,y)=1+tg'(y)>0.\no
\eeq
By assumption \eqref{0-7} or \eqref{0-8}, we have that near $x=0$,

(i) for $0\le t< 1$, $\f{\p x}{\p y}(t,y)>0$;

(ii) $\f{\p x}{\p y}(1,y)\ge 0$, and only at $y=0$, $\f{\p x}{\p y}(1,y)=0$.

Thus for $0\le t\le 1$, one can get a function
$y=y(t,x)$ satisfying \eqref{2-3} such that the solution to \eqref{0-1} is
\beq\label{2-5}
u(t,x)=u_0(y(t,x)).
\eeq
On the other hand, one can compute that for $0\le t<1$,
\beq\label{2-6}
\left\{
  \begin{array}{l}
   \f{\p y}{\p t}=-\f{g(y)}{1+tg'(y)},\\
   \f{\p y}{\p x}=\f{1}{1+tg'(y)}.
  \end{array}
\right.
\eeq
This means that as $(t,x)$ tends to $(1-,0)$, then $y(t,x)\rightarrow 0$ and
$|\p_x y(t,x)|\rightarrow +\i$.

\bigskip
Let $\ve>0$ be a sufficiently small constant.
Under assumption \eqref{0-2}, it is easy to check that for $1<t<1+\ve$ and $y$ near $0$, there exist two roots of $\p_y x(t,y)=0$
with respect to the variable $y$, which are denoted by
$\eta_-(t)$ and $\eta_+(t)$ with $\eta_-(t)<\eta_+(t)$. Set $x_\pm(t)=x(t,\eta_\pm(t))$, we then have
\begin{itemize}
  \item for $x<x_+(t)$ ($x>x_-(t)$ resp.) and equality \eqref{2-3}, there exists a unique root denoted by $y_-(t,x)$ ($y_+(t,x)$ resp.).
  \item for $x=x_+(t)$ ($x=x_-(t)$ resp.)  and equality \eqref{2-3}, there exist two roots denoted by $y_-(t,x)<\eta_+(t)$ ($\eta_-(t)<y_+(t,x)$ resp.).
  \item for $x_+(t)<x<x_-(t)$  and equality \eqref{2-3}, there exist three roots denoted by $y_-(t,x)<y_0(t,x)<y_+(t,x)$.
\end{itemize}

Set
\ben
\O_-&=&\{(t,x):1< t<1+\ve,\ x<x_-(t)\}\no\\
\O_+&=&\{(t,x):1< t<1+\ve,\ x>x_+(t)\}\no\\
\O_0&=&\{(t,x):1< t<1+\ve,\ x_+(t)<x<x_-(t)\}.\no
\een

Under \eqref{0-7}, we derive some properties of $\eta_\pm(t)$ and $x_\pm(t)$ near the blowup point $(1, 0)$.

\begin{lemma}\label{lem 3-1}
There exists an $\ve>0$ sufficiently small such that\\
(1) $\eta_\pm(t)\in C^{2k+1}(1,1+\ve)\cap C^\f{1}{2k}[1,1+\ve)$ admit the following expansion
\beq\label{3-1}
\eta_\pm(t)=\pm(2k+1)^{-\f{1}{2k}}(t-1)^{\f{1}{2k}}
-\f{g^{(2k+2)}(0)}{2k(2k)!}(2k+1)^{-\f{2k+1}{k}}(t-1)^{\f{1}{k}}+o((t-1)^{\f{1}{k}});
\eeq
(2) $x_\pm(t)=x(t,\eta_\pm(t))\in C^{2k+1}(1,1+\ve)\cap C^\f{2k+1}{2k}[1,1+\ve)$ are the envelopes of the characteristic lines
 \eqref{2-3} which form a cusp at $(1,0)$, meanwhile,
\beq\label{3-2}
x_\pm(t)=\mp 2k(2k+1)^{-\f{2k+1}{2k}}(t-1)^\f{2k+1}{2k}
+\f{g^{(2k+2)}(0)}{(2k+2)!}(2k+1)^{-\f{k+1}{k}}(t-1)^{\f{k+1}{k}}+o((t-1)^{\f{k+1}{k}}).
\eeq
\end{lemma}
\begin{proof}
(1) Note that $\eta_\pm(t)$ are the solutions of
\beq\label{3-3}
1+tg'(y)=-(t-1)+(2k+1)ty^{2k}+t r'(y)=0.
\eeq
This immediately yields $\eta_\pm\in C^{2k+1}(1,1+\ve]$ by the implicit function theorem.
For $t\to 1+$, set
\beq\label{3-4}
s=(t-1)^\f{1}{2k},\ z=\f{y}{s}.
\eeq
Then \eqref{3-3} becomes
\beq\label{3-5}
F(s,z)\triangleq (1+s^{2k})[(2k+1)z^{2k}+s^{-2k}r'(sz)]-1=0.
\eeq
Since $r'(sz)=O(s^{2k+1})$ for $s$ near $0$,
$F(0,z^0_\pm)=0$ holds for $z^0_\pm=\pm(2k+1)^{-\f{1}{2k}}$. By direct computation, we have that
\ben
\label{3-6} \p_s F(s,z)&=&2k(2k+1)s^{2k-1}z^{2k}-2k s^{-2k-1}r'(sz)+(s^{-2k}+1)z r''(sz),\\
\label{3-7} \p_z F(s,z)&=&(1+s^{2k})[2k(2k+1)z^{2k-1}+s^{-2k+1}r''(sz)].
\een
Together with $r(sz)=\f{g^{(2k+2)}(0)}{(2k+2)!}(sz)^{2k+2}+o(s^{2k+2})$, this yields
\ben
\label{3-8} \p_s F(0,z^0_\pm)&=&\pm\f{g^{(2k+2)}(0)}{(2k+1)!}(2k+1)^{-\f{2k+1}{2k}},\\
\label{3-9} \p_z F(0,z^0_\pm)&=&\pm 2k(2k+1)^\f{1}{2k}\neq0.
\een
By the implicit function theorem, for small $\ve>0$ there exist
\beq\label{3-10}
z=z_\pm(s)\in C^{2k+1}[0,\ve]
\eeq
such that $F(s,z_\pm(s))=0$ and
\beq\label{3-11}
z_\pm(s)=z^0_\pm-\f{g^{(2k+2)}(0)}{2k(2k)!}(2k+1)^{-\f{2k+1}{k}}s+o(s).
\eeq
Therefore, \eqref{3-1} is shown and then $\eta_\pm(t)\in C^{\f{1}{2k}}[1,1+\ve]$.

\bigskip

(2) By \eqref{0-7} and \eqref{2-3}, we have
$$
x_\pm(t)=x(t,\eta_\pm(t))=-(t-1)\eta_\pm(t)+t\eta^{2k+1}_\pm(t)+t r(\eta_\pm(t)).
$$
Together with \eqref{3-1}, this yields $x_\pm(t)\in C^{2k+1}(1,1+\ve]\cap C^\f{2k+1}{2k}[1,1+\ve]$
and the expansion \eqref{3-2}. In addition, due to $\f{\p}{\p y}x(t,\eta_\pm(t))=0$, then for $t\in[1,1+\ve]$,
\beq\label{3-12}
\f{d}{dt}x_\pm(t)=\f{\p}{\p t}x(t,\eta_\pm(t))=g(\eta_\pm(t)).
\eeq
This means that the tangent direction of $x=x_\pm(t)$ coincides with the characteristic speed of \eqref{2-3} at $(t,x_\pm(t))$.
Consequently, the proof of (2) is finished.
%
\end{proof}

Under \eqref{0-8}-\eqref{0-9}, we have
\begin{lemma}\label{lem 4-1}
For $\eta_\pm(t),\ x_\pm(t),\ y_\pm(t,x)$ and $y_0(t,x)$, we can deduce the following properties
for small $\ve>0$:\\
(1) $\eta_\pm(t)\in C^\i(1,1+\ve]\cap C[1,1+\ve]$ with $\eta_\pm(\tau)=\pm|\ln(t-1)|^{-\f{1}{p}}+O(\f{\ln|\ln(t-1))|}{|\ln(t-1)|})$.\\
(2) $x_\pm(t)=x(t,\eta_\pm(t))\in C^\i(1,1+\ve]\cap C^{1}[1,1+\ve]$ are the envelopes of the characteristic lines
and form a cusp at $(1,0)$. Moreover we have the expansion $x_\pm(t)=\mp(t-1)|\ln(t-1)|^{-\f{1}{p}}+O(\f{(t-1)\ln|\ln(t-1)|}{|\ln(t-1)|})$.\\
(3) For any $t\in (1,1+\ve]$, $y_-(t,\cdot)$ is an increasing function from $(-\i,x_-(t)]$ onto $(-\i,\eta_-(t)]$; $y_0(t,\cdot)$
is a decreasing function from $[x_+(t),x_-(t)]$ onto $[\eta_-(t),\eta_+(t)]$; $y_+(t,\cdot)$ is an increasing function from
$[x_+(t),+\i)$ onto $[\eta_+(t),+\i)$. Moreover, $y_m(t,x)\in C^\i(\O_m)\cap C(\bar\O_m)$, where $m=-,+,0$.
\end{lemma}
\begin{proof}
(1) Set $\tau=t-1$ for $t\geq1$. Note that $\eta_\pm(t)$ are defined for small $\tau>0$ and  are the solutions of the following equation
\beq\label{4-1}
1+tg'(y)=\left(-\tau+\f{1}{|y|^p} e^{-|y|^{-p}}\right)+\f{\tau}{|y|^p} e^{-|y|^{-p}}+(\tau+1)e^{-|y|^{-p}}r_1(y)=0,
\eeq
where $r_1(y)=\f{1}{p}+\f{y r_0(y)}{|y|^{p+2}}+r_0'(y)=O(|y|^{
\min\{-p+1,0\}})$. Denote $\o=|\ln\tau|^{-1}$ and $z=|y|^{-p}-(\o^{-1}-\ln \o)$, then \eqref{4-1} becomes
\beq\label{4-2}
F(\tau,z)\triangleq -1+\left(1+\o z-\o\ln \o\right)e^{-z}+\o e^{-z}\left(
\tau\left(z+\o^{-1}-\ln \o\right)+(\tau+1)r_1(y)\right)=0.
\eeq
Obviously, $F(0,0)=0$. In addition, by direct computation, we have that for small $|y|$,
$$
r'_1(y)=-\f{p+1}{|y|^{p+2}}r_0(y)+\f{yr'_0(y)}{|y|^{p+2}}+r''_0(y)=O(|y|^{-p})
$$
and
$$
\f{\p y}{\p z}=-\f{y|y|^{p}}{p}.
$$
Thus one can check that
\beq
\f{\p F}{\p z}=\o e^{-z}-\left(1+z\o-\o\ln\o\right)e^{-z}-\o e^{-z}\left(
\tau\left(z+\o^{-1}-\ln \o-1\right)+(\tau+1)(r_1(y)-r'_1(y)
\f{\p y}{\p z})\right).\no
\eeq
This yields $\f{\p F}{\p z}(0,0)=-1$ by $|y|\lesssim \o^\f{1}{p}$. Since $F(\tau,z)$ is continuous and has the continuous partial derivative $\f{\p F}{\p z}$ near $(0,0)$, it follows from the implicit function theorem that there exists a continuous function $z=z(\tau)$ near $z=0$ to satisfy $F(\tau,z)=0$. This deduces  $\eta_\pm(t)=\pm(|\ln\tau|+\ln|\ln\tau|+z(\tau))^{-\f{1}{p}}=\pm|\ln\tau|^{-\f{1}{p}}+O(\f{\ln|\ln\tau|}{|\ln\tau|})$. On the other hand, for $\tau>0$,
we have $\f{\p x}{\p y}(t,y)\neq 0$ and $g'\in C^\i$. Then by the implicit function theorem, $\eta_\pm(t)\in C^\i(1,1+\ve]$ hold.

(2) By \eqref{0-8} and \eqref{2-3}, we have that for small $\tau>0$,
$$
x_\pm(t)=-\tau\eta_\pm(t)+t(\f{\eta_\pm(t)}{p}e^{-|\eta_\pm(t)|^{-p}}+r_0(\eta_\pm(t)))=\tau[-\eta_\pm(t)
+t(\f{\eta_\pm(t)}{p}+o(\eta_\pm(t)))e^{-|O(\f{\ln|\ln\tau|}{|\ln\tau|})|}],
$$
which derives $x_\pm(t)\in C^\i(1,1+\ve]\cap C[1,1+\ve]$. On the other hand, it holds that for $t\in (1,1+\ve]$,
\ben
\f{d}{dt}x_\pm(t)&=&\f{\p}{\p t}x(t,\eta_\pm(t))+\f{\p}{\p y}x(t,\eta_\pm(t))\f{d}{dt}\eta_\pm(t)\no\\
&=&g(\eta_\pm(t)),\no
\een
which means that the tangent direction of $x=x_\pm(t)$ is same as the characteristic speed of \eqref{2-3} at
$(t,x_\pm(t))$. In addition, at the point $(1,0)$, one has
\ben
x'_\pm(1)&=&\lim_{t\rightarrow1_+}\f{x_\pm(t,\eta_\pm(t))-0}{t-1}\no\\
&=&\lim_{t\rightarrow1_+}[-\eta_\pm(t)+t(\f{\eta_\pm(t)}{p}+o(\eta_\pm(t)))e^{O(\f{\ln|\ln\tau|}{|\ln\tau|})}]\no\\
&=&0\no\\
&=&g(0).\no
\een
Hence we finish the proof of (2).

(3) For any fixed $t\in (1, 1+\ve]$, due to
$$
\f{\p}{\p y}x(t,y)\left\{
                    \begin{array}{ll}
                      >0, & \text{ for }y\in(-\i,\eta_-(t))\cup(\eta_+(t),+\i),\\
                      =0, & \text{ for }y=\eta_\pm(t),\\
                      <0, & \text{ for }y\in(\eta_-(t),\eta_+(t)),
                    \end{array}
                  \right.
$$
then by the inverse function theorem, $y_m(t,\cdot)$ with $m=-,+,0$ are well defined and satisfy the corresponding monotonicity.
Moreover, $y_+(t,\cdot)\in C^\i(x_+(t),+\i)\cap C[x_+(t),+\i)$, $y_0(t,\cdot)\in C^\i(x_+(t),x_-(t))\cap C[x_+(t),x_-(t)]$
and $y_-(t,\cdot)\in C^\i(-\i,x_-(t))\cap C(-\i,x_-(t)]$.

On the other hand, because of $\p_y x(t,x)\neq 0$ for $(t,x)\notin \{x=x_\pm(t)\}$, thus
it follows from the implicit function theorem
that $y_m(t,x)\in C^\i(\O_m)$, $m=-,+,0$. For the continuity of $y_m(t,x)$ in $\bar \O_m$, we take $y_+(t,x)$ as an example.
By $x_+(t)\in C^1([1,1+\ve])$, we then get
$$
|y_+(\bar t,x_+(\bar t))-y_+(t,x)|=|y_+(\bar t,x_+(\bar t))-y_+(t,x_+(t))|+|y_+(t,x_+(t))-y_+(t,x)|\rightarrow 0
$$
as $(t,x)\rightarrow(\bar t,x_+(\bar t))$ for $\bar t\in[1,1+\ve]$ and $(t,x)\in \O_+$. Thus $y_+(t,x)\in C^\i(\O_+)\cap C(\bar\O_+)$
holds.\\
\end{proof}

To study the formation of shock wave and the regularity of the resulting shock $x=\vp(t)$ to equation \eqref{0-1},
it is required to study the properties of $y_\pm(t,x)$ for $(t,x)$ in the cusp domain $\O_0$. Under assumption
\eqref{0-5}, motivated by \cite{Le94}, we take the following change of the variables
\beq\label{3-15}
\tau=t-1,\ s=\tau^\f{1}{2k},\ \mu=\f{y}{s},\ \la=\f{x}{s^{2k+1}},
\eeq
and will establish the behavior of $y_\pm(t,x)$ near $(1,0)$ in some sub-domain of $\O_0$.

\begin{lemma}\label{lem 3-2}
For small $\ve>0$, under  assumption
\eqref{0-7}, there exists some constant $\dl>0$ such that for $(s,\la)\in\{0\le s\le \ve,\ |\la|\le \dl\}$,
$(s,\la)\rightarrow s^jy_\pm(t,x)$ are of $C^{j+2}$ for $j=-1,0,1,\ldots,2k$ and $y_{\pm}(t,x)$ admit the
following expansions
\ben
\label{3-16}y_+(t,x)&=&s(1+\f{\la}{2k}-\f{g^{(2k+2)}(0)}{2k(2k+2)!}s)+O(s^3+s\la^2),\\
\label{3-17}y_-(t,x)&=&s(-1+\f{\la}{2k}-\f{g^{(2k+2)}(0)}{2k(2k+2)!}s)+O(s^3+s\la^2).
\een
\end{lemma}
\begin{proof}
Let
\beq\label{3-18}
h(y)\triangleq \f{r(y)}{y^{2k+1}}=\int_0^1 \f{(1-\th)^{2k}}{(2k)!}g^{(2k+1)}(\th y)d\th-1.
\eeq
Then $h(y)\in C^{p-2k-1}$ and $h(0)=0$. Furthermore,
\beq\label{3-19}
y h'(y)=\f{y}{(2k)!}\int_0^1(1-\th)^{2k}\th g^{(2k+2)}(\th y)d\th
=-\f{1}{(2k)!}\int_0^1g^{(2k+1)}(\th y)(1-\th)^{2k-1}[1-(2k+1)\th]~d\th.
\eeq
This derives $y h'(y)\in C^1$. Similarly, $y^j h^{(j)}(y)\in C^{1}$ holds for $j=2,\ldots,2k+1$.
Divided by $s^{2k}$, \eqref{2-3} becomes
\beq\label{3-20}
G(s,\la,\mu)\triangleq -\mu+(1+s^{2k})\mu^{2k+1}+(1+s^{2k})\mu^{2k+1}h(s\mu)-\la=0.
\eeq
For $s=\la=0$, by $G(0,0,\mu)=-\mu+\mu^{2k+1}=-\mu(1-\mu^{2k})=0$
we get the roots $\mu^0_\pm=\pm1$ and $\mu^0_c=0$. Note that
\beq\label{3-21}
\p_\mu G(s,\la,\mu)=-1+(2k+1)(1+s^{2k})\mu^{2k}+(1+s^{2k})((2k+1)\mu^{2k}h(s\mu)+\mu^{2k+1}sh'(s\mu)).
\eeq
Then
\beq\label{3-22}
\p_\mu G(0,0,\pm 1)=2k\neq0.
\eeq
By the implicit function theorem, there exist functions $\mu=\mu_\pm(s,\la)\in C^{1}$ near $(s,\la)=(0,0)$ such that
\beq\label{3-23}
G(s,\la,\mu_\pm(s,\la))=0,\ \mu_\pm(0,0)=\pm 1,
\eeq
and then $s^{-1}y_\pm\in C^{1}$. On the other hand, due to
\ben
\label{3-24}\p_s G(s,\la,\mu)&=&2ks^{2k-1}\mu^{2k}+2ks^{2k-1}\mu^{2k+1}h(s\mu)+(1+s^{2k})\mu^{2k+2}h'(s\mu),\\
\label{3-25}\p_\la G(s,\la,\mu)&=&-1,
\een
then
\ben
\label{3-26}\p_s G(0,0,\pm 1)&=&h'(0)
=\f{g^{(2k+2)}(0)}{(2k)!}\int_0^1(1-\th)^{2k}\th d\th=\f{g^{(2k+2)}(0)}{(2k+2)!} \\
\label{3-27}\p_\la G(0,0,\pm 1)&=&-1.
\een
It follows from \eqref{3-22}, \eqref{3-26} and \eqref{3-27} that
\beq\label{3-28}
\p_s \mu_\pm(0,0)=-\f{g^{(2k+2)}(0)}{2k(2k+2)!}, \quad \p_\la \mu_\pm(0,0)=\f{1}{2k}.
\eeq
Consequently, the expansions \eqref{3-16} and \eqref{3-17} are shown.

We next prove $(s,\la)\rightarrow y_\pm\in C^{2}$. By
\beq\label{3-29}
s\p_s G(s,\la,\mu_\pm)+\p_\mu G(s,\la,\mu_\pm)(s\p_s\mu_\pm(s,\la))=0
\eeq
and
\beq\label{3-30}
s\p_s G(s,\la,\mu_\pm(s,\la))
=2ks^{2k}\mu_\pm^{2k}+2ks^{2k}\mu_\pm^{2k+1}h(s\mu_\pm)
+(1+s^{2k})\mu_\pm^{2k+1}(s\mu_\pm h'(s\mu_\pm))\in C^{1},
\eeq
we then have $s\p_s\mu_\pm(s,\la)\in C^{1}$.
Thus from $\p_s y_\pm=\mu_\pm+s\p_s\mu_\pm$, one can see that
$(s,\la)\to y_\pm$ is of $C^{1}$.
In addition, by $\p_\la y_\pm=s\p_\la\mu_\pm(s,\la)$ and similar computation,
we can get $y_\pm\in C^{2}$ with respect to $s$ and $\la$.

Note that
\ben\label{3-31}
s^j\p^j_s G(s,\la,\mu_\pm(s,\la))&=&\vG(s,\la,s\p_s\mu_\pm,\ldots,s^{j-1}\p^{j-1}_s\mu_\pm,
h(s\mu_\pm),(s\mu_\pm)h'(s\mu_\pm),\ldots,(s\mu_\pm)^jh^{(j)}(s\mu_\pm))\no\\
&&+\p_\mu G(s,\la,\mu_\pm)(s^j\p^j_s\mu_\pm(s,\la)),\ j=2,3,\ldots,2k+1,\no
\een
where $\vG$ is a polynomial with respect to its arguments. Then
$s^j\p^j_s\mu_\pm(s,\la)\in C^{1}$ for $j=2,3,\ldots,2k+1$ by induction.
Similarly, $s^j\p^m_
\la\p^{j-m}_s \mu_\pm(s,\la)\in C^1$ for $1\le m\le j\le 2k+1$. Consequently,  the proof
of $(s,\la)\rightarrow s^jy_\pm(t,x)\in C^{j+2}$ for $j=-1,0,1,\ldots,2k$ is completed.
\end{proof}

Under assumption \eqref{0-8}, we now study the asymptotic behavior of $y_\pm(t,x)$ near $(1,0)$.
In this case, we take the following change of the variables
\beq\label{4-4}
\tau=t-1,\ s=|\ln\tau|^{-\f{1}{p}},\ \la=\f{x}{s\tau},\ \mu=\f{y}{s}.
\eeq
Then we obtain

\begin{lemma}\label{lem 4-2}
Under assumption \eqref{0-8}, there exist some small constants $\ve$, $\dl>0$ such that for $(t,x)\in\O_0\triangleq\{1\le t\le 1+\ve,\ -\dl s\tau< x< \dl s\tau\}$, $y_\pm\in C^{1+p}$ hold for the variables $s$ and $\la$. Furthermore,
\ben
\label{4-5}y_+(t,x)&=&s\left(1+\f{\ln p}{p}s^p+\f{s^p\la}{p}\right)+O\left(s^{\min\{p+2,2p+1\}}+s^{p+1}|\la|^2\right),\\
\label{4-6}y_-(t,x)&=&s\left(-1-\f{\ln p}{p}s^p+\f{s^p\la}{p}\right)+O\left(s^{\min\{p+2,2p+1\}}+s^{p+1}|\la|^2\right).
\een
\end{lemma}
\begin{proof}
By divided by $s\tau$, \eqref{2-3} can be written as follows
\beq\label{4-7}
G(s,\la,\mu)\triangleq \mu(-1+\f{1}{p}e^{s^{-p}(1-|\mu|^{-p})})+\f{\mu}{p}e^{-s^{-p}|\mu|^{-p}}
+\f{e^{s^{-p}}+1}{s}e^{-s^{-p}|\mu|^{-p}}r_0(s\mu)-\la=0.
\eeq
Below we assume $\mu\neq0$. Without loss of generality, one can assume
$\mu>0$ (corresponding to the case of $y_+(t,x)$). We divide the proof of Lemma \ref{lem 4-2}
 into the two cases of $p\ge1$ and $p\in(0,1)$.

\bigskip

\noindent {\it Case 1. $p\geq1$.}

Set $\zeta=s^{-p}(1-\mu^{-p})-\ln p$ and $\mu=\left(1-s^p(\zeta+\ln p)\right)^{-\f{1}{p}}$. Then \eqref{4-7} becomes
\beq\label{4-8-1}
F_1(s,\la,\zeta)\triangleq G(s,\la,\mu)=\mu(-1+e^\zeta)+\mu e^{\zeta-s^{-p}}+\f{pe^{\zeta}(1+e^{-s^{-p}})}{s}r_0(s\mu)-\la.
\eeq
Obviously $F_1(0,0,0)=0$ and $\ds\lim_{s\to0+,\zeta\to0}\mu=1$. Note that
\beq
\p_s\mu=s^{p-1}(\zeta+\ln p)\left(1-s^p(\zeta+\ln p)\right)^{-\f{1}{p}-1},\quad
\p_\zeta\mu=\f{s^p}{p}\left(1-s^p(\zeta+\ln p)\right)^{-\f{1}{p}-1}.\no
\eeq
Thanks to $p\ge1$, we have that $\p_s\mu$ is bounded. On the other hand,
\ben
\p_s F_1&=&ps^{-p-1}\mu e^{\zeta-s^{-p}}+\left(-1+e^\zeta+e^{\zeta-s^{-p}}\right)\p_s\mu\no\\
&&+\f{pe^\zeta\left(e^{-s^{-p}}(ps^{-p}-1)-1\right)}{s^2}r_0(s\mu)+\f{pe^{\zeta}(1+e^{-s^{-p}})}{s}
r'_0(s\mu)\left(\mu+s\p_s\mu\right),\no\\
\p_\la F_1&=&-1,\no\\
\p_\zeta F_1&=&\mu\left(e^\zeta+e^{\zeta-s^{-p}}\right)
+\left(-1+e^\zeta+e^{\zeta-s^{-p}}\right)\p_\zeta\mu
+\f{pe^{\zeta}(1+e^{-s^{-p}})}{s}\left(r_0(s\mu)+s r'_0(s\mu)\p_\zeta\mu\right).\no
\een
This derives $F_1\in C^1$ and
\beq
\p_s F_1(0,0,0)=\f{p}{2}r''_0(0),\quad\p_\la F_1(0,0,0)=-1,\quad\p_\zeta F_1(0,0,0)=1.
\eeq
Thus by the implicit function theorem, one can obtain that there exists a unique function $\zeta(s,\la)\in C^1$ satisfying $F_1(s,\la,\zeta(s,\la))=0$ and admitting the following expansion
\beq\label{4-9-1}
\zeta(s,\la)=\f{p}{2}r''_0(0)s+\la+O\left(s^2+\la^2\right).
\eeq
At this time, we get
\ben\label{4-10-1}
\mu(s,\la)&=&\left(1-s^p(\zeta+\ln p)\right)^{-\f{1}{p}}\no\\
&=&\left(1-s^p(\ln p+\f{p}{2}r''_0(0)s+\la+O\left(s^2+\la^2\right)\right)^{-\f{1}{p}}\no\\
&=&1+\f{\ln p}{p}s^p+\f{s^p\la}{p}+O\left(s^{p+1}+s^p|\la|^2\right)
\een
and $s^{l}\mu\in C^{l+p}$, $l=0,1$.

If we consider the case of $\mu<0$, then by the same method, one can obtain the expansion
\beq\label{4-11-1}
\mu(s,\la)=-1-\f{\ln p}{p}s^p+\f{s^p\la}{p}+O\left(s^{p+1}+s^p|\la|^2\right).
\eeq

\bigskip

\noindent {\it Case 2. $0<p<1$.}

Set $\o=s^{p}$, $\zeta=\o^{-1}(1-\mu^{-p})-\ln p$ and $\mu=\left(1-\o(\zeta+\ln p)\right)^{-\f{1}{p}}$. Then \eqref{4-7} becomes
\beq\label{4-8-2}
F_2(\o,\la,\zeta)\triangleq G(s,\la,\mu)=\mu(-1+e^\zeta)+\mu e^{\zeta-\o^{-1}}+\f{pe^{\zeta}(1+e^{-\o^{-1}})}{\o^\f{1}{p}}r_0(\o^{\f{1}{p}}\mu)-\la.
\eeq
Obviously $F_2(0,0,0)=0$ and $\ds\lim_{\o\to0+,\ \zeta\to0}\mu=1$. Note that
\beq
\p_\o\mu=\f{\zeta+\ln p}{p}\left(1-\o(\zeta+\ln p)\right)^{-\f{1}{p}-1},\quad
\p_\zeta\mu=\f{\o}{p}\left(1-\o(\zeta+\ln p)\right)^{-\f{1}{p}-1}.\no
\eeq
On the other hand,
\ben
\p_\o F_2&=&\o^{-2}\mu e^{\zeta-\o^{-1}}+\left(-1+e^\zeta+e^{\zeta-\o^{-1}}\right)\p_\o\mu\no\\
&&+\f{e^\zeta\left(e^{-\o^{-1}}(p\o^{-1}-1)-1\right)}{\o^{\f{1}{p}+1}}r_0(\o^\f{1}{p}\mu)
+\f{pe^{\zeta}(1+e^{-\o^{-1}})}{\o^\f{1}{p}}r'_0(\o^\f{1}{p}\mu)\cdot\o^{\f{1}{p}-1}\left(\f{1}{p}\mu+\o\p_\o\mu\right),\no\\
\p_\la F_2&=&-1,\no\\
\p_\zeta F_2&=&\mu\left(e^\zeta+e^{\zeta-\o^{-1}}\right)
+\left(-1+e^\zeta+e^{\zeta-\o^{-1}}\right)\p_\zeta\mu
+\f{pe^{\zeta}(1+e^{-\o^{-1}})}{\o^\f{1}{p}}\left(r_0(\o^\f{1}{p}\mu)+\o^\f{1}{p} r'_0(\o^\f{1}{p}\mu)\p_\zeta\mu\right).\no
\een
Then $F_2\in C^1$ and
\beq
\p_\o F_2(0,0,0)=0,\quad\p_\la F_2(0,0,0)=-1,\quad\p_\zeta F_2(0,0,0)=1.
\eeq
Thus by the implicit function theorem, one can obtain that there exists a unique function $\zeta(\o,\la)\in C^1$ satisfying $F_2(\o,\la,\zeta(\o,\la))=0$ and admitting such an expansion
\beq\label{4-9-2}
\zeta(\o,\la)=\la+O\left(\o^2+\la^2\right).
\eeq
This yields
\ben\label{4-10-2}
\mu(s,\la)&=&\left(1-s^p\left(\zeta(s^p,\la)+\ln p\right)\right)^{-\f{1}{p}}\no\\
&=&\left(1-s^p(\ln p+\la+O\left(s^{2p}+\la^2\right)\right)^{-\f{1}{p}}\no\\
&=&1+\f{\ln p}{p}s^p+\f{s^p\la}{p}+O\left(s^{2p}+s^p|\la|^2\right)
\een
and $s^{l}\mu\in C^{l+p}$ for $l=0,1$.

If we consider the case of $\mu<0$, then by the same method, one can obtain the expansion
\beq\label{4-11-2}
\mu(s,\la)=-1-\f{\ln p}{p}s^p+\f{s^p\la}{p}+O\left(s^{2p}+s^p|\la|^2\right).
\eeq


Consequently, we complete the proof of \eqref{4-5} and \eqref{4-6}.

\end{proof}

\section{Proof of Theorem \ref{main result1}}

We now construct the shock curve $x=\vp(t)$ of \eqref{0-1} in $\O_0$.
It follows from the Rankine-Hugoniot condition that
\beq\label{3-13}
\left\{
  \begin{array}{l}
    \vp'(t)=\f{f(u_0(y_+(t,\vp(t))))-f(u_0(y_-(t,\vp(t))))}{u_0(y_+(t,\vp(t)))-u_0(y_-(t,\vp(t)))},\\
    \vp(1)=0.
  \end{array}
\right.
\eeq
This, together with the mean-value theorem, yields
\beq\label{3-14}
g(y_+(t,\vp(t)))<\vp'(t)<g(y_-(t,\vp(t))),
\eeq
which means that the entropy condition on $x=\vp(t)$ holds.
Denote
\beq\label{3-32}
a(x,y)\triangleq
\left\{
  \begin{array}{ll}
    \f{f(u_0(x))-f(u_0(y))}{u_0(x)-u_0(y)}, & \text{ if }x\neq y, \\
    g(x), & \text{ if }x=y.
  \end{array}
\right.
\eeq
Under assumption \eqref{0-7}, it is easy to verify $a(x,y)\in C^{2k+2}(\bR^2)$ with $a(x,y)=-\f{1}{2}(x+y)+b(x,y)$,
where $b(x,y)=O(x^2+y^2)\in C^{2k+2}$.

\bigskip


\begin{lemma}\label{lem 3-3}
Under assumption \eqref{0-7}, for \eqref{3-13} and small $\ve>0$, there exists a solution $x=\vp(t)$ on $[1,1+\ve)$
to satisfy\\
(1) $\vp(t)$ is a $C^{2}$ function on $s\in[0,\ve)$, where $s=(t-1)^{\f{1}{2k}}$;\\
(2) $\vp(t)\in C^{\f{k+1}{k}}[1,1+\ve)$.
\end{lemma}

\begin{proof}
(1) Set $\la(s)=\f{\vp(t)}{s^{2k+1}}$. Then \eqref{3-13} becomes
\beq\label{3-34}
\left\{
  \begin{array}{l}
s\la'(s)+(2k+2)\la(s)=s d(s,\la(s)),\\
\la(0)=0,
  \end{array}
\right.
\eeq
where
\beq\label{3-35}
d(s,\la)=-\f{k(\mu_+(s,\la)+\mu_-(s,\la)-\la)}{s}+\f{2k}{s^2}b(s\mu_+(s,\la),s\mu_-(s,\la)).
\eeq

By the proof procedure of Lemma \ref{lem 3-2}, we have $s^jd(s,\la)\in C^{j}$ for $j=0,1,2$.
Then by the same analysis in \cite{Le94}, there exists a unique solution $\la(s)\in C^{1}[0,\ve)$
to \eqref{3-34} and further $s\la'(s)\in C^{1}$. Due to $(s\la(s))'=s\la'(s)+\la(s)$,
then $s\la(s)\in C^{2}$, and $s^2\la'(s)=s^2d-(2k+1)s\la\in C^2$.
Therefore, $s\rightarrow\vp'(t)=\f{1}{2k}[s^2\la'(s)+(2k+1)s\la(s)]$ is
a $C^2$ function.

(2) Let $s\rightarrow0+$ in \eqref{3-34}, we have
$$
\la'(0)=\f{d(0,0)}{2k+3}=\f{g^{(2k+2)}(0)}{(2k+3)!}+\lim_{s\rightarrow0_+}\f{b(s\mu_+,s\mu_-)}{s^2}=O(1).
$$
Therefore $\la(s)=O(s)$ and then $\vp(t)=O(s^{2k+2})=O((t-1)^\f{k+1}{k})\in C^{\f{k+1}{k}}[1,1+\ve)$.

\end{proof}

{\bf Remark 3.1.} {\it The regularity of $\vp(t)$ in Lemma \ref{lem 3-3} is optimal. Indeed, we consider
the following problem
$$
\left\{
  \begin{array}{ll}
    &\p_t u+\p_x(\f{1}{2}u^2)=0, \\
    &\ds u(0,x)=-x+x^{2k+1}+|x|^{2k+2+\ve},\ \ve>0.
  \end{array}
\right.
$$
In this case, we have $g(x)=-x+x^{2k+1}+x^{2k+2}$ and
\ben
y_+(t,\vp(t))&=&(t-1)^\f{1}{2k}(1+\f{\vp(t)}{2k(t-1)^\f{2k+1}{2k}}-\f{(t-1)^\f{1}{2k}}{2k})+o((t-1)^\f{1}{k})\no\\
y_-(t,\vp(t))&=&(t-1)^\f{1}{2k}(-1+\f{\vp(t)}{2k(t-1)^\f{2k+1}{2k}}-\f{(t-1)^\f{1}{2k}}{2k})+o((t-1)^\f{1}{k}).
\een
It follows from Rankine-Hugoniot condition that
\ben
\vp'(t)&=&-\f{y_-(t,\vp(t))+y_+(t,\vp(t))}{2}\no\\
&=&\f{\vp(t)}{k(t-1)}-\f{(t-1)^\f{1}{k}}{k}+o((t-1)^\f{1}{k}).\no
\een
This derives $\vp(t)\in C^{\f{k+1}{k}}[1,1+\ve)$ which is optimal.}

\begin{lemma}\label{lem 3-4}
Under assumption \eqref{0-7}, for any $c\in(-\f{2k}{(2k+1)^{\f{1}{2k}+1}},+\i)$, there exist $\ve=\ve(c)>0$
and $\dl=\dl(c)>0$
such that for $(s,\la)\in\{0< s<\ve,c-\dl<\la<c+\dl\}$,
$(s,\la)\rightarrow y_+(t,x)$ has the expansion
\beq
\label{3-36}y_+(t,x)=s\left(\mu_c-\f{\mu_c^{2k+2}g^{(2k+2)}(0)}{(-1+\mu_c^{2k})(2k+2)!}s
+\f{\la-c}{-1+\mu_c^{2k}}\right)+O(s^3+s(\la-c)^2),
\eeq
and for $(s,\la)\in\{0<s<\ve,-c-\dl<\la<-c+\dl\}$,
$(s,\la)\rightarrow y_-(t,x)$ has the expansion
\beq
\label{3-37}y_-(t,x)=s\left(-\mu_c-\f{\mu_c^{2k+2}g^{(2k+2)}(0)}{(-1+\mu_c^{2k})(2k+2)!}s
+\f{\la+c}{-1+\mu_c^{2k}}\right)+O(s^3+s(\la+c)^2),
\eeq
where $\mu_c$ is the unique solution in $(\f{1}{(2k+1)^\f{1}{2k}},+\i)$ of the equation
\beq\label{3-38}
G(0,c,\mu)=-\mu+\mu^{2k+1}-c=0.
\eeq
\end{lemma}
\begin{proof}
By \eqref{3-21}, \eqref{3-24} and \eqref{3-25},
we have that
\ben
\label{3-39}\p_\mu G(0,\pm c,\pm\mu_c)&=&-1+(2k+1)\mu_c^{2k}>0,\\
\label{3-40}\p_s   G(0,\pm c,\pm\mu_c)&=&\mu_c^{2k+2}h'(0)
=\f{\mu_c^{2k+2}g^{(2k+2)}(0)}{(2k+2)!} ,\\
\label{3-41}\p_\la G(0,\pm c,\pm\mu_c)&=&-1.
\een
Then by the implicit function theorem one has that there is a unique function
$\mu_\pm(s,\la)$ near $(0,\pm c)$ satisfying $G(s,\la,\mu_\pm(s,\la))\equiv0$ and
admitting the expansion
\beq\label{3-42}
\mu_\pm(s,\la)=\pm\mu_c-\f{1}{-1+\mu_c^{2k}}
\left(\f{\mu_c^{2k+2}g^{(2k+2)}(0)}{(2k+2)!}s-(\la-c)\right)+O(s^2+(\la-c)^2).
\eeq
Thus \eqref{3-36} and \eqref{3-37} are proved.
\end{proof}

To study the asymptotic behavior of $y_\pm$ near the $x-$axis,
we now take the following transform
\beq\label{3-43}
\xi=x^\f{1}{2k+1},\ \eta=\f{t-1}{\xi^{2k}},\ \nu=\f{y}{\xi}.
\eeq
Under assumption \eqref{0-6}, by divided $\xi^{2k+1}$, \eqref{2-3} then becomes
\beq\label{3-44}
H(\eta,\xi,\nu)\triangleq -\eta\nu+(1+\xi^{2k}\eta)\nu^{2k+1}+(1+\xi^{2k}\eta)\nu^{2k+1}h(\xi\nu)-1=0.
\eeq
We now have

\begin{lemma}\label{lem 3-5}
Under assumption \eqref{0-6}, for small $\dl>0$, there exists some small constant $\ve>0$ such that for $(\eta,\xi)\in\{ |\eta|\le \ve,\ 0<\xi<\dl\}$, we can get the expansion of $y_+(t,x)$ on $(\xi,\eta)$
\beq
\label{3-45}y_+(t,x)=\xi\left(1+\f{\eta}{2k+1}
-\f{g^{(2k+2)}(0)}{(2k+2)!}\xi\right)+O(\eta^2\xi+\xi^3),
\eeq
and for $(\eta,\xi)\in\{|\eta|\le \ve,\ -\dl<\xi<0\}$, we can get
\beq
\label{3-46}y_-(t,x)=\xi\left(1-\f{\eta}{2k+1}
-\f{g^{(2k+2)}(0)}{(2k+2)!}\xi\right)+O(\eta^2\xi+\xi^3).
\eeq
\end{lemma}
\begin{proof}
It follows from direct computation that $H(0,0,1)=0$ and
\ben
\label{3-47}\p_\nu  H&=&-\eta+(2k+1)(1+\xi^{2k}\eta)\nu^{2k}+(2k+1)(1+\xi^{2k}\eta)\nu^{2k}h(\xi\nu)
+(1+\xi^{2k}\eta)\nu^{2k+1}\xi h'(\xi\nu),\\
\label{3-48}\p_\eta H&=&-\nu+\xi^{2k}\nu^{2k+1}+\xi^{2k}\nu^{2k+1}h(\xi\nu),\\
\label{3-49}\p_\xi  H&=&2k\xi^{2k-1}\eta\nu^{2k+1}+2k\xi^{2k-1}\eta\nu^{2k+1}h(\xi\eta)+(1+\xi^{2k}\eta)\nu^{2k+1}\nu h'(\xi\nu).
\een
Then
\ben
\label{3-50}\p_\nu  H(0,0,1)&=&2k+1,\\
\label{3-51}\p_\eta H(0,0,1)&=&-1,\\
\label{3-52}\p_\xi  H(0,0,1)&=&h'(0)=\f{g^{(2k+2)}(0)}{(2k+2)!}.
\een
By the implicit function theorem, there exists a unique solution $\nu=\nu(\eta,\xi)$
of \eqref{3-44} near $(\eta,\xi)=(0,0)$ such that
\beq
\label{3-53}\nu(\eta,\xi)=1+\f{\eta}{2k+1}
-\f{g^{(2k+2)}(0)}{(2k+2)!}\xi+O(\eta^2+\xi^2).
\eeq
Therefore, we obtain \eqref{3-45} for $\xi>0$. Analogously,  \eqref{3-46} holds for $\xi<0$.
\end{proof}

Next we consider the asymptotic behavior of $y(t,x)$ near the blowup point $(1, 0)$ in the domain $\{(t,x):\ t<1\}$.
Without confusions, we still use the same notation as for $t>1$. Note that through each point in $\{t<1\}$,
there exists a unique characteristic line. By taking the following transform similar to \eqref{3-15}
\beq\label{3-54}
\tau=1-t,\ s=\tau^{\f{1}{2k}},\ \mu=\f{y}{s},\ \la=\f{x}{s^{2k+1}},
\eeq
and then by divided $s^{2k+1}$ on two sides of  \eqref{2-3}, then \eqref{2-3} becomes
\beq\label{3-55}
G(s,\la,\mu)\triangleq \mu+(1-s^{2k})\mu^{2k+1}+(1-s^{2k})\mu^{2k+1}h(s\mu)-\la=0.
\eeq

\begin{lemma}\label{lem 3-6}
For each $c\in\bR$, there exist $\ve=\ve(c),\dl=\dl(c)>0$
such that for $(s,\la)\in\{0< s<\ve,c-\dl<\la<c+\dl\}$,
$(s,\la)\rightarrow y(t,x)$ has the expansion
\beq
\label{3-56}y(t,x)=s\left(\mu_c-\f{\mu_c^{2k+2}g^{(2k+2)}(0)}{(1+(2k+1)\mu_c^{2k})(2k+2)!}s
-\f{\la-c}{1+(2k+1)\mu_c^{2k}}\right)+O(s^3+s(\la-c)^2),
\eeq
where $\mu_c$ is the unique solution of the equation
\beq\label{3-57}
G(0,c,\mu)=\mu+\mu^{2k+1}-c=0.
\eeq
\end{lemma}
\begin{proof}
It follows from direct computation that
\ben
\label{3-58}\p_\mu G(s,\la,\mu)&=&1+(2k+1)(1 -s^{2k})\mu^{2k}+(2k+1)(1-s^{2k})\mu^{2k+1}sh'(s\mu),\\
\label{3-59}\p_s   G(s,\la,\mu)&=&-2ks^{2k-1}\mu^{2k+1}-2ks^{2k-1}\mu^{2k+1}h(s\mu)
+(1-s^{2k})\mu^{2k+2}h'(s\mu),\\
\label{3-60}\p_\la G(s,\la,\mu)&=&-1.
\een
Then
\ben
\label{3-61}\p_\mu G(0,c,\mu_c)&=&1+(2k+1)\mu_c^{2k}>0,\\
\label{3-62}\p_s   G(0,c,\mu_c)&=&\mu_c^{2k+2}h'(0)=\f{\mu_c^{2k+2}g^{(2k+2)}(0)}{(2k+2)!},\\
\label{3-63}\p_\la G(0,c,\mu_c)&=&-1.
\een
By the implicit function theorem, there exists a $\mu(s,\la)$ near $(c,\mu_c)$ satisfying
\beq
\label{3-64}
\mu(s,\la)=\mu_c-\f{1}{1+(2k+1)\mu_c^{2k}}
\left(\f{\mu_c^{2k+2}g^{(2k+2)}(0)}{(2k+2)!}s-(\la-c)\right)+O(s^2+(\la-c)^2),
\eeq
from which we can deduce \eqref{3-56}.
\end{proof}
\bigskip

We start to prove Theorem \ref{main result1}.

\vspace{0.3cm}

\noindent{\it Proof of Theorem \ref{main result1}:}

\noindent (1) By Lemma \ref{lem 3-3}, $\vp(t)\in C^\f{k+1}{k}[1,1+\ve)$ and
$u\in C^1((1,1+\ve)\times\bR)\setminus \{x=\vp(t)\})$ for small $\ve>0$ have been shown.

\noindent (2)  Let $\dl,\ \ve>0$ be the constants obtained in Lemma \ref{lem 3-5} and denote
\ben
\label{3-65}\O_{x,+}&=&B\cap\{(t,x):\ 0<x<\dl^{2k+1},\ |t-1|<\ve x^\f{2k}{2k+1}\},\\
\label{3-66}\O_{x,-}&=&B\cap\{(t,x):\ -\dl^{2k+1}<x<0,\ |t-1|<\ve (-x)^\f{2k}{2k+1}\},\\
\label{3-67}\O_{0}&=&B\cap\{(t,x):\ t<1,\ |x|^\f{1}{2k+1}<\f{2}{\ve^\f{1}{2k}}(1-t)^\f{1}{2k}\}.
\een
Let $c_0\in(0,\f{2k}{(2k+1)^{1+\f{1}{2k}}})$ be some fixed constant and denote
\ben
\label{3-68}\O_{t,+}&=&B\cap\{(t,x):\ -c_0(t-1)^\f{1}{2k}<x^\f{1}{2k+1}<\f{2}{\ve^\f{1}{2k}}(t-1)^\f{1}{2k}\},\\
\label{3-69}\O_{t,-}&=&B\cap\{(t,x):\ -\f{2}{\ve^\f{1}{2k}}(t-1)^\f{1}{2k}<x^\f{1}{2k+1}<c_0(t-1)^\f{1}{2k}\}.
\een
It is easy to see that for $(t,x)\in\O_{x,+}\cup\O_{t,+}$, $u(t,x)=u_0(y_+(t,x))$;
for $(t,x)\in\O_{x,-}\cup\O_{t,-}$, $u(t,x)=u_0(y_-(t,x))$; for $(t,x)\in\O_{0}$, $u(t,x)=u_0(y(t,x))$. By Heine-Borel property of compactness, there exist $\{c_{j,\pm},\dl_{j,\pm}=\dl_{j,\pm}(c_{j,\pm}),\ve_{j,\pm}=\ve_{j,\pm}(c_{j,\pm})\}_{j=1}^n$ and $\{c_{j,0},\dl_{j,0}=\dl_{j,0}(c_{j,0}),\ve_{j,0}=\ve_{j,0}(c_{j,0})\}_{j=1}^n$ such that
\beq
\label{3-70}
\O_{t,+}\subset\cup_{j=1}^n\O_{t,+}^j,\ \O_{t,-}\subset\cup_{j=1}^n\O_{t,-}^j,\ \O_{0}\subset\cup_{j=1}^n\O_{0}^j,
\eeq
where
\ben
\label{3-71} \O_{t,+}^j&=&\{(t,x):\ 0<(t-1)^{\f{1}{2k}}<\ve_{j,+},\ c_{j,+}-\dl_{j,+}<\f{x}{(t-1)^\f{2k+1}{2k}}<c_{j,+}+\dl_{j,+}\},\\
\label{3-72} \O_{t,-}^j&=&\{(t,x):\ 0<(t-1)^{\f{1}{2k}}<\ve_{j,-},\ c_{j,-}-\dl_{j,-}<\f{x}{(t-1)^\f{2k+1}{2k}}<c_{j,-}+\dl_{j,-}\},\\
\label{3-73} \O_{0}^j&=&\{(t,x):\ 0<(1-t)^{\f{1}{2k}}<\ve_{j,0},\ c_{j,+}-\dl_{j,+}<\f{x}{(1-t)^\f{2k+1}{2k}}<c_{j,+}+\dl_{j,+}\},
\een
and these domains satisfy the corresponding properties in Lemma \ref{lem 3-4} and \ref{lem 3-6}.

Set $B=\{(t,x):\ 0<\sqrt{(t-1)^2+x^2}< \rho\}$, and choose $\rho>0$ sufficiently small such that
\beq\label{3-74}
B=\O_{x,+}\cup\O_{x,-}\cup\O_{t,+}\cup\O_{t,-}\cup\O_{0}.
\eeq
We now establish the behaviors of $u$ and its derivatives near $(1,0)$.
It suffices to only consider this in the domains $\O_{x,+}$, $\O_{t,+}^j$ and  $\O_{0}^j$
since the other cases can be treated analogously.

For $(t,x)\in\O_{x,+}$, we have
\beq
\label{3-75}
|u(t,x)-u(1,0)|=|u_0(y_+(t,x))|\lesssim |y_+(t,x)|\lesssim x^\f{1}{2k+1};
\eeq
for $(t,x)\in\O_{t,+}^j$,
\beq
\label{3-76}
|u(t,x)-u(1,0)|=|u_0(y_+(t,x))|\lesssim |y_+(t,x)|\lesssim (t-1)^\f{1}{2k};
\eeq
and for $(t,x)\in\O_{0}^j$,
\beq
\label{3-77}
|u(t,x)-u(1,0)|=|u_0(y(t,x))|\lesssim |y(t,x)|\lesssim (1-t)^\f{1}{2k}.
\eeq
Therefore \eqref{1.1} is obtained.

Let's turn to prove the estimates \eqref{1.2} and \eqref{1.3}.
Note that
\beq\label{3-78}
\left\{
  \begin{array}{ll}
    &\f{\p u}{\p x}=u'_0(y(t,x))\f{\p y}{\p x}(t,x)=\f{u'_0(y(t,x))}{1+tg'(y(t,x))},\\
    &\f{\p u}{\p t}=u'_0(y(t,x))\f{\p y}{\p t}(t,x)=-\f{u'_0(y(t,x))g(y(t,x))}{1+tg'(y(t,x))}.
  \end{array}
\right.
\eeq

For $(t,x)\in\O_{t,+}^j$, by \eqref{3-36} in Lemma \ref{lem 3-4} we have
\ben\label{3-79}
1+tg'(y_+(t,x))&=&-s^{2k}+(2k+1)(1+s^{2k})y_+^{2k}(t,x)+O(y_+^{2k+1}(t,x))\no\\
&=&(-1+(2k+1)\mu_{c_{j,+}}^{2k})s^{2k}+O(s^{2k+1}+s^{2k}|\la-c_{j,+}|)\no\\
&\gtrsim& s^{2k}\no\\
&=&(t-1)\no\\
&\gtrsim&|t-1|+|x|^\f{2k}{2k+1},
\een
where $s=t-1$, and the fact of $-1+(2k+1)\mu_{c_{j,+}}^{2k}>0$ has been used.

For $(t,x)\in\O_{x,+}$, by \eqref{3-45} in Lemma \ref{lem 3-5} we have
\ben\label{3-80}
1+tg'(y(t,x))&=&-\eta\xi^{2k}+(2k+1)(1+\eta\xi^{2k})y^{2k}(t,x)+O(y^{2k+1}(t,x))\no\\
&=&\xi^{2k}+O(\xi^{2k+1}+\eta\xi^{2k})\no\\
&\gtrsim& \xi^{2k}\no\\
&=&x^\f{2k}{2k+1}\no\\
&\gtrsim&|t-1|+|x|^\f{2k}{2k+1},
\een
where $\xi=x^\f{1}{2k+1}$.

For $(t,x)\in\O_{0,+}^j$, by \eqref{3-56} in Lemma \ref{lem 3-6} we arrive at
\ben\label{3-81}
1+tg'(y(t,x))&=&s^{2k}+(2k+1)(1+s^{2k})y^{2k}(t,x)+O(y^{2k+1}(t,x))\no\\
&=&(1+(2k+1)\mu_{c_{j,0}}^{2k})s^{2k}+O(s^{2k+1}+s^{2k}|\la-c_{j,0}|)\no\\
&\gtrsim& s^{2k}\no\\
&=&(1-t)\no\\
&\gtrsim&|t-1|+|x|^\f{2k}{2k+1},
\een
where $s=1-t$, and the fact of $1+(2k+1)\mu_{c_{j,0}}^{2k}>0$ has been used.

Therefore, $1+tg'(y(t,x))\gtrsim|t-1|+|x|^\f{2k}{2k+1}$ holds for $(t,x)\in B$.
In light of \eqref{3-78} and the fact of $g(y(t,x))\sim y(t,x)$ in $B$,
\eqref{1.2}-\eqref{1.3} hold and thus the proof of Theorem \ref{main result1}
is completed.

\rightline{$\square$}

\section{Proof of Theorem \ref{main result2}}

By the characteristics method, we can define $u_\pm(t,x)=u_0(y_\pm(t,x))$. By (\ref{Rankine-Hugoniot Condition}),
the shock curve $x=\vp(t)$ satisfies
\beq\label{4-3}
\bec
&\vp'(t)=\f{f(u_0(y_+(t,\vp(t))))-f(u_0(y_-(t,\vp(t))))}{u_0(y_+(t,\vp(t)))-u_0(y_-(t,\vp(t)))},\\
&\vp(1)=0.
\eec
\eeq

Denote
\beq\label{4-14}
a(x,y)\triangleq \bec
&\ds\f{f(u_0(x))-f(u_0(y))}{u_0(x)-u_0(y)},\text{ if }x\neq y,\\
&g(x),\text{ if }x=y.
\eec
\eeq
By \eqref{0-8}, we have $a(x,y)$ in $C^\i(\bR^2)$ and
\beq\label{4-15}
a(x,y)=-\f{1}{2}(x+y)+b(x,y),
\eeq
where $b(x,y)=b(y,x)$ and $b(x,y)=O(x^2+y^2)\in C^\i$.

We now study the regularity of the shock wave $x=\vp(t)$ as a function of $s=|\ln(t-1)|^{-\f{1}{p}}$.

\begin{lemma}\label{lem 4-3}
Under assumption \eqref{0-8}, for \eqref{4-3} and small $\ve>0$, there exists a solution $x=\vp(t)$ on $[1,1+\ve)$
such that\\
(1) $s\rightarrow\vp(t)$ is of $C^{1}$ on $[0,\ve)$;\\
(2) $x=\vp(t)$ is of $C^1$ on $[1,1+\ve]$ with the behavior $\vp(t)=O(s^2\tau)$.
\end{lemma}
\begin{proof}
(1) Set $\la(s)=\f{\vp(t)}{s\tau}$. Then
$$
\f{d\vp}{dt}(t)=s\left(\f{s^{1+p}}{p}\f{d\la(s)}{ds}+(\f{s^p}{p}+1)\la(s)\right).
$$
Substituting this into \eqref{4-3} yields
\ben\label{4-16}
&&\f{s^{1+p}}{p}\f{d\la(s)}{ds}+(\f{s^p}{p}+1)\la(s)\no\\
&=&\f{1}{s}a(s\mu_+(s,\la(s)),s\mu_-(s,\la(s)))\no\\
&=&-\f{1}{2}\left(\mu_+(s,\la(s))+\mu_-(s,\la(s))\right)+\f{1}{s}b(s\mu_+(s,\la(s)),s\mu_-(s,\la(s)))\\
&=&-\f{s^p\la}{p}+O\left(s+s^{2p+1}|\la|+s^{p+1}|\la|^2\right).
\een
By \eqref{4-5} and \eqref{4-6}, we have
\beq\label{4-17}
d(s,\la)\triangleq \f{1}{s}a(s\mu_+(s,\la),s\mu_-(s,\la))+\f{s^p}{p}\la
=O\left(s+s^{2p+1}|\la|+s^{p+1}|\la|^2\right).
\eeq
Moreover, $s^ld(s,\la)\in C^{l+p}$ holds for $l=0,1$, which is derived by Lemma \ref{lem 4-2} and
\beq\label{4-18}
\f{d(sd(s,\la(s)))}{ds}=O(1+s^{2p}|\la|+s^{2p+1}|\la'|+s^p|\la|^2+s^{p+2}|\la'|^2).
\eeq
In addition, \eqref{4-3} in $(s,\la)$ can be written as
\beq\label{4-19}
\bec
&\f{s^{1+p}}{p}\f{d\la(s)}{d s}+(\f{2}{p}s^p+1)\la(s)=d(s,\la(s)),\\
&\la(0)=0.
\eec
\eeq
This yields
\beq\label{4-20}
\la(s)=ps^{-2}\int_0^s \o^{1-p}e^{s^{-p}-\o^{-p}}d(\o,\la(\o))d\o.
\eeq
It follows from direct computation that
\ben\label{4-21}
|\la(s)|&\leq&s^{-2}\int_0^ss^2e^{s^{-p}}\o^{-1-p}e^{-\o^{-p}}|d(\o,\la(\o))|d\o)\no\\
&\lesssim& (s+s^{2p+1}\|\la\|_{L^\i[0,s]}+s^{p+1}\|\la\|^2_{L^\i[0,s]})\int_0^se^{s^{-p}}\o^{-1-p}e^{-\o^{-p}}d\o\no\\
&\lesssim& s+s^{2p+1}\|\la\|_{L^\i[0,s]}+s^{p+1}\|\la\|^2_{L^\i[0,s]}.\no
\een
Thus $\|\la\|_{L^\i[0,s]}\leq Cs$ for $s\in(0,\ve]$ and small $\ve>0$. By the analogous
computation, we can apply the contraction mapping theorem to prove that there exists a continuous solution
$\la$ to the integral equation \eqref{4-20}. From \eqref{4-20}, we have
\ben\label{4-22}
\la'(s)&=&ps^{-1-p}d(s,\la(s))-2ps^{-3}\int_0^s\o^{1-p}e^{s^{-p}-\o^{-p}}d(\o,\la(\o))d\o
-ps^{-p-3}e^{s^{-p}}\int_0^s\o^2d(\o,\la(\o))de^{-\o^{-p}}\no\\
&=&-2ps^{-3}\int_0^s\o^{1-p}e^{s^{-p}-\o^{-p}}d(\o,\la(\o))d\o+ps^{-p-3}e^{s^{-p}}\int_0^s
e^{-\o^{-p}}\big(\o d(\o,\la(\o))+\o\f{d(\o d(\o,\la(\o)))}{d\o}\big)d\o\no.
\een
This derives
\ben\label{4-23}
|\la'(s)|&\lesssim&s^{-3}\int_0^ss^2\o^{-1-p}e^{s^{-p}-\o^{-p}}|d(\o,\la(\o))|d\o+s^{-p-3}\int_0^ss^{2+p}\o^{-1-p}e^{s^{-p}-\o^{-p}}|d(\o,\la(\o))|d\o\no\\
&&+s^{-p-3}e^{s^{-p}}\int_0^ss^{2+p}\o^{-1-p}e^{-\o^{-p}}|\f{d(\o d(\o,\la(\o)))}{d\o}|d\o\no\\
&\lesssim&1+\f{\|\la\|_{L^\i[0,s]}}{s}+s^{-1}e^{s^{-p}}\int_0^s\o^{-1-p}e^{-\o^{-p}}(s+s^{2}\|\la'\|_{L^\i[0,s]}+\|\la\|^2_{L^\i[0,s]})d\o\no\\
&\lesssim&1+\f{\|\la\|_{L^\i[0,s]}}{s}+s\|\la'\|_{L^\i[0,s]},\no
\een
and then $\la'(s)\in C[0,\ve]$ can be shown. In addition, by $\vp(t)=s\tau\la(s)$ and $s=|\ln\tau|^{-\f{1}{p}}$,
then $\vp(t)=O(s^2\tau)$ holds.

\end{proof}

{\bf Remark 4.1.} {\it The regularity of $\vp(t)$ in Lemma \ref{lem 4-3} is optimal. Indeed, we consider Burgers' equation
$$
\left\{
  \begin{array}{ll}
    &\ds\f{\p u}{\p t}+\f{\p}{\p x}(\f{1}{2}u^2)=0, \\
    &\ds u(0,x)=-x+\f{1}{p}e^{-|x|^{-p}}\left(x+x^2\right),\ p>0.
  \end{array}
\right.
$$
In this case, $g(x)=u_0(x)=-x+\f{1}{p}e^{-|x|^{-p}}\left(x+x^2\right)$. On the other hand, \eqref{4-7} can be written as
$$
F(s,\la,\zeta)\triangleq G(s,\la,\mu)=\mu(-1+e^\zeta)+\mu e^{\zeta-s^{-p}}+s\mu^{2}e^{\zeta}(1+e^{-s^{-p}})-\la,
$$
which derives $\f{\p F}{\p s}|_{s=\la=\zeta=0}=1$. So we have that
\ben
\label{4-24-1}y_+(t,x)&=&|\ln(t-1)|^{-\f{1}{p}}(1+\f{\ln p}{p}|\ln(t-1)|^{-1})+\f{\vp(t)}{p(t-1)|\ln(t-1)|}+\f{1}{p}|\ln(t-1)|^{-1-\f{2}{p}}+o(|\ln(t-1)|^{-1-\f{2}{p}}),\no\\
\label{4-24-2}y_-(t,x)&=&-|\ln(t-1)|^{-\f{1}{p}}(1+\f{\ln p}{p}|\ln(t-1)|^{-1})+\f{\vp(t)}{p(t-1)|\ln(t-1)|}+\f{1}{p}|\ln(t-1)|^{-1-\f{2}{p}}+o(|\ln(t-1)|^{-1-\f{2}{p}}).\no
\een
It follows from  Rankine-Hugoniot condition that
\ben
\vp'(t)&=&-\f{y_-(t,\vp(t))+y_+(t,\vp(t))}{2}\no\\
&=&-\f{\vp(t)}{p(t-1)|\ln(t-1)|}-\f{1}{p}|\ln(t-1)|^{-1-\f{2}{p}}+o(|\ln(t-1)|^{-1-\f{2}{p}}).\no
\een
This means that $\vp(t)=O((t-1)|\ln(t-1)|^{-\f{2}{p}})$ is optimal.}

\vspace{0.3cm}

\begin{lemma}\label{lem 4-4}
Under assumption \eqref{0-7}, for any $c\in(-1,+\i)$, there exist $\ve=\ve(c),\dl=\dl(c)>0$
such that for $(s,\la)\in\{0<s<\ve,c-\dl<\la<c+\dl\}$,
$(s,\la)\rightarrow y_+(t,x)$ has the expansion
\beq
\label{4-25}y_+(t,x)=s\left(1+\f{\ln(c+1)+\ln p}{p}s^p+\f{s^p(\la-c)}{p(c+1)}\right)+O_c(s^{\min\{p+2,2p+1\}}+s^{p+1}|\la-c|^2),
\eeq
and for $(s,\la)\in\{0<s<\ve,-c-\dl<\la<-c+\dl\}$,
$(s,\la)\rightarrow y_-(t,x)$ has the expansion
\beq
\label{4-26}y_-(t,x)=s\left(-1-\f{\ln(c+1)+\ln p}{p}s^p+\f{s^p(\la+c)}{p(c+1)}\right)+O_c(s^{\min\{p+2,2p+1\}}+s^{p+1}|\la+c|^2).
\eeq
\end{lemma}
\begin{proof}
Similarly to Lemma \ref{lem 4-2}, we first consider the case of $p\geq1$. By taking $\la=c>-1$ and $s=0$ in \eqref{4-8-1},
we have the solution $\zeta_{c}=\ln(c+1)$ and $\mu_{c}=1$. Furthermore, direct computation yields
\ben
\label{4-27-0-1}\p_s F_1(0,\zeta_c,c)&=&c\left(\ln(c+1)+\ln p\right)\dl_p^1+\f{p(c+1)}{2}r''_0(0),\\
\label{4-27-1}\p_\zeta F_1(0,\zeta_c,c)&=&c+1,\\
\label{4-28-1}\p_\la F_1(0,\zeta_c,c)&=&-1,
\een
where $\dl_p^1=\left\{\begin{array}{cc}1,&\ p=1,\\0,&\ p>1.\end{array}\right.$ By the implicit function theorem, we have
\beq
\label{4-29-1}\zeta(s,\la)=\ln(c+1)-\left(c\left(\ln(c+1)+\ln p\right)\dl_p^1+\f{p(c+1)}{2}r''_0(0)\right)s+\f{\la-c}{c+1}+O_c(s^{2}+|\la-c|^2),
\eeq
and then
\beq
\label{4-30-1}
\mu_+(s,\la)=(1-s^p(\zeta+\ln p))^{-\f{1}{p}}=1+\f{\ln(c+1)+\ln p}{p}s^p+\f{s^p(\la-c)}{p(c+1)}+O_c(s^{p+1}+s^p|\la-c|^2).
\eeq
from which \eqref{4-25} follows.

For $p\in(0,1]$, recalling $\o=s^p$ and then taking $\o=0$, $\la=c$ and $\zeta=\zeta_c=\ln(c+1)$ in \eqref{4-8-2},
one can arrive at
\ben
\label{4-27-0-2}\p_\o F_2(0,\zeta_c,c)&=&\f{c\left(\ln(c+1)+\ln p\right)}{p},\\
\label{4-27-2}\p_\zeta F_2(0,\zeta_c,c)&=&c+1,\\
\label{4-28-2}\p_\la F_2(0,\zeta_c,c)&=&-1,
\een
and then by the implicit function theorem we have
\beq
\label{4-29-2}\zeta(s,\la)=\tilde\zeta(\o,\la)=\ln(c+1)+\f{c\left(\ln(c+1)+\ln p\right)}{p}\o+\f{\la-c}{c+1}+O_c(s^{2}+|\la-c|^2).
\eeq
Thus
\beq\label{4-30-2}
\mu_+(s,\la)=\tilde\mu_+(\o,\la)=(1-\o(\zeta+\ln p))^{-\f{1}{p}}=1+\f{\ln(c+1)+\ln p}{p}\o+\f{\o(\la-c)}{p(c+1)}+O_c(\o^{2}+\o|\la-c|^2)
\eeq
and \eqref{4-25} can be obtained.

On the other hand for $\la=-c$ and $\mu_{-c}=-1$, \eqref{4-26} can be proved by the same method.
\end{proof}

\bigskip

Next we consider the behavior of $y(t,x)$ near $x-$axis. Note that for $y>0$,
\beq\label{4-31}
x=\xi e^{-\xi^{-p}}
\eeq
is a monotonically increasing function of $\xi$ from $[0,+\i)$ to $[0,+\i)$. Then there exists a unique inverse function $h(x)$ of \eqref{4-31} satisfying that for $x>0$ sufficiently small,
\beq\label{4-32}
h(x)=|\ln x|^{-\f{1}{p}}+O(|\ln x|^{-1-\f{1}{p}}\ln|\ln x|).
\eeq
Define
\beq\label{4-33}
\xi=\left\{
      \begin{array}{ll}
        h(|x|), & x>0, \\
        -h(|x|), & x<0,
      \end{array}
    \right.
\nu=\f{y}{\xi},\ \eta=\f{(t-1)h(|x|)}{|x|}.
\eeq

\begin{lemma}\label{lem 4-5}
Under assumption \eqref{0-7}, there exist some constants $\ve$, $\dl>0$ small enough such that for $(\eta,\xi)\in\{0<\xi<\dl,-\ve<\eta<\ve\}$,
$(\eta,\xi)\rightarrow y_+(t,x)$ has the expansion
\beq
\label{4-34-1}y_+(t,x)=\xi\left(1+\f{\ln p}{p}\xi^p+\f{1}{p}\xi^p\eta\right)+O(\xi^{\min\{p+2,2p+1\}}+\xi^{p+1}\eta^2),
\eeq
and for $(\eta,\xi)\in\{-\dl<\xi<0,-\ve<\eta<\ve\}$,
$(\eta,\xi)\rightarrow y_+(t,x)$ has the expansion
\beq
\label{4-34-2}y_-(t,x)=\xi\left(1+\f{\ln p}{p}(-\xi)^p+\f{1}{p}(-\xi)^p\eta\right)+O((-\xi)^{\min\{p+2,2p+1\}}+(-\xi)^{p+1}\eta^2).
\eeq
\end{lemma}
\begin{proof}
We only consider the case of $x>0$ and then $y=y_+(t,x)>0$ since the other case can be treated analogously.
By $x=\xi e^{-\xi^{-p}}$, $y=\xi\nu$, $t-1=\eta e^{-\xi^{-p}}$ and \eqref{0-8}, \eqref{2-3} becomes
\beq\label{4-35}
H(\eta,\xi,\nu)\triangleq-\eta\nu+\f{\nu}{p}e^{-\xi^{-p}\nu^{-p}}\left(\eta+e^{\xi^{-p}}\right)
+\f{e^{-\xi^{-p}\nu^{-p}}\left(\eta+e^{\xi^{-p}}\right)}{\xi}r_0(\xi\nu)-1=0.
\eeq
Similarly to Lemma \ref{lem 4-2}, we divide the proof procedure into two cases of $p\ge1$ and $0<p<1$.
Firstly we consider $p\ge1$. Set $\th=\xi^{-p}(1-\nu^{-p})-\ln p$ and $\nu=\left(1-\xi^{p}(\th+\ln p)\right)^{-\f{1}{p}}$. Then \eqref{4-35} becomes
\beq\label{4-36-1}
J_1(\eta,\xi,\th)\triangleq H(\eta,\xi,\nu)=-\eta\nu+\left(\nu e^\th-1\right)
+\eta\nu e^{\th-\xi^{-p}}+\f{p e^\th\left(\eta e^{-\xi^{-p}}+1\right)}{\xi}r_0(\xi\nu).
\eeq
Note $J_1(0,0,0)=0$, and
\beq\label{4-37-1}
\p_\xi\nu=\xi^{p-1}(\th+\ln p)(1-\xi^p(\th+\ln p))^{-\f{1}{p}-1},\ \p_\th\nu=\f{1}{p}\xi^p(1-\xi^p(\th+\ln p))^{-\f{1}{p}-1}
\eeq
are bounded near $\xi=0$ and $\th=0$ by $p\ge1$. Due to
\ben\label{4-38-1}
\p_\xi J_1&=&(e^\th-\eta)\p_\xi\nu+\eta e^{\th-\xi^{-p}}(\p_\xi\nu+p\nu\xi^{-p-1})\\
&&-\f{pe^\th}{\xi^2}(\eta e^{-\xi^{-p}}(1-p\xi^{-p})+1)r_0(\xi\nu)+\f{p e^\th\left(\eta e^{-\xi^{-p}}+1\right)}{\xi}r'_0(\xi\nu)(\nu+\xi\p_\xi\nu),\\
\p_\th J_1&=&e^\th\nu+(e^\th-\eta)\p_\th\nu+\eta e^{\th-\xi^{-p}}(\nu+\p_\th\nu)+\f{p e^\th\left(\eta e^{-\xi^{-p}}+1\right)}{\xi}(r_0(\xi\nu)+\xi r'_0(\xi\nu)\p_\th\nu),\\
\p_\eta J_1&=&-\nu+e^{\th-\xi^{-p}}(\nu+\f{p}{\xi}r_0(\xi\nu)),
\een
we then obtain
\beq\label{4-39-1}
\p_\xi J_1(0,0,0)=\dl_1^p\ln p+\f{p}{2}r''_0(0),\ \p_\th J_1(0,0,0)=1,\ \p_\eta J_1(0,0,0)=-1.
\eeq
Thus by the implicit function theorem, one can deduce that there exists a unique function $\th=\th(\eta,\xi)$ near $(\eta,\xi)=(0,0)$ satisfying
\beq\label{4-40-1}
\th(\eta,\xi)=-\left(\dl_1^p\ln p+\f{p}{2}r''_0(0)\right)\xi+\eta+O(\xi^2+\eta^2).
\eeq
Recalling $\nu=\left(1-\xi^p(\th+\ln p)\right)^{-\f{1}{p}}$, we then have
\beq\label{4-41-1}
\nu(\eta,\xi)=1+\f{\ln p}{p}\xi^p+\f{1}{p}\xi^p\eta+O\left(\xi^{p+1}+\xi^p\eta^2\right).
\eeq

\bigskip

For $p\in(0,1)$, set $\vs=\xi^p$ and then
$\nu=(1-\vs(\th+\ln p))^{-\f{1}{p}}$.
In this case, \eqref{4-35} becomes
\beq\label{4-36-2}
J_2(\eta,\vs,\th)\triangleq H(\eta,\xi,\nu)=-\eta\nu+\left(\nu e^\th-1\right)
+\eta\nu e^{\th-\vs^{-1}}+\f{p e^\th\left(\eta e^{-\vs^{-1}}+1\right)}{\vs^\f{1}{p}}r_0(\vs^{\f{1}{p}}\nu).
\eeq

By $J_2(0,0,0)=0$ and
\beq\label{4-37-1}
\p_\vs\nu=\f{1}{p}(\th+\ln p)(1-\vs(\th+\ln p))^{-\f{1}{p}-1},\ \p_\th\nu=\f{\vs}{p}(1-\vs(\th+\ln p))^{-\f{1}{p}-1},
\eeq
under  assumption \eqref{0-8}, we have that $J_2(\eta,\xi,\th)\in C^\f{1}{p}$ and
\ben\label{4-38-1}
\p_\xi J_2&=&(e^\th-\eta)\p_\vs\nu+\eta e^{\th-\vs^{-1}}(\p_\vs\nu+\nu\vs^{-2})\\
&&-\f{e^\th}{\vs^{\f{1}{p}+1}}(\eta e^{-\vs^{-1}}(1-p\vs^{-1})+1)r_0(\vs^\f{1}{p}\nu)+\f{p e^\th\left(\eta e^{-\vs^{-1}}+1\right)}{\vs^\f{1}{p}}r'_0(\vs^\f{1}{p}\nu)(\f{1}{p}\vs^{\f{1}{p}-1}\nu+\vs^\f{1}{p}\p_\vs\nu),\\
\p_\th J_2&=&e^\th\nu+(e^\th-\eta)\p_\th\nu+\eta e^{\th-\vs^{-1}}(\nu+\p_\th\nu)+\f{p e^\th\left(\eta e^{-\vs^{-1}}+1\right)}{\vs^\f{1}{p}}(r_0(\vs^\f{1}{p}\nu)+\vs^\f{1}{p} r'_0(\vs^\f{1}{p}\nu)\p_\th\nu),\\
\p_\eta J_2&=&-\nu+e^{\th-\vs^{-1}}(\nu+\f{p}{\vs^\f{1}{p}}r_0(\vs^\f{1}{p}\nu)).
\een
This yields
\beq\label{4-39-2}
\p_\vs  J_2(0,0,0)=\f{\ln p}{p},\ \p_\th J_2(0,0,0)=1,\ \p_\eta J_2(0,0,0)=-1.
\eeq
By the implicit function theorem, we know that there exists a unique solution $\th=\th(\eta,\vs)$
for $(\eta,\vs)$ near $(0,0)$, which satisfies $\th(0,0)=0$ and
\beq\label{4-40-2}
\th(\eta,\vs)=-\f{\ln p}{p}\vs+\eta+O\left(\vs^2+\eta^2\right).
\eeq
By $\xi=\vs^\f{1}{p}$, then
\ben\label{4-41-2}
\nu(\eta,\xi)&=&(1-\vs(\th+\ln p))^{-\f{1}{p}}\no\\
&=&\left(1-\vs\ln p-\vs\eta+O\left(\vs^2+\eta^2\right)\right)^{-\f{1}{p}}\no\\
&=&1+\f{\ln p}{p}\xi^p+\f{1}{p}\xi^p\eta+O(\xi^{2p}+\xi^p\eta^2).
\een
Therefore we finish the proof of \eqref{4-34-1}.

For $x<0$, we can transform \eqref{2-3} to $H(\eta,-\xi,-\nu)=0$. Analogously, we can obtain \eqref{4-34-2} about $y_-(t,x)$
and $x<0$.
\end{proof}

\bigskip

Next we consider the behavior of $y(t,x)$ for $t<1$ near $(1,0)$. Without of confusion, we still denote
\beq
\label{4-41}
\tau=1-t,\ s=|\ln \tau|^{-\f{1}{p}},\ \la=\f{x}{s\tau},\ \mu=\f{y}{s},
\eeq
as in the case of $t>1$. By divided $s\tau$, \eqref{2-3} then becomes
\beq\label{4-42}
G(s,\la,\mu)\triangleq\mu(1+\f{1}{p}e^{s^{-p}(1-|\mu|^{-p})})-\f{\mu}{p}e^{-s^{-p}|\mu|^{-p}}
+\f{e^{-|\mu|^{-p}s^{-p}}(e^{s^{-p}}-1)}{s}r_0(s\mu)-\la=0.
\eeq

\begin{lemma}\label{lem 4-6}
Under assumption \eqref{0-7}, we have that

\noindent (1) for any $c>1$, there exist $\ve=\ve(c),\dl=\dl(c)>0$
such that for $(s,\la)\in\{0<s<\ve,1<c-\dl<\la<c+\dl\}$,
$(s,\la)\rightarrow y(t,x)$ has the expansion
\beq
\label{4-43-1}y(t,x)=s\left(1+\f{\ln(c-1)+\ln p}{p}s^p+\f{s^p(\la-c)}{p(c-1)}\right)+O(s^{\min\{p+2,2p+1\}}+s^{p+1}|\la-c|^2).
\eeq\\
(2) for any $0\le c<1$, there exist $\ve=\ve(c),\dl=\dl(c)>0$
such that for $(s,\la)\in\{0<s<\ve,c-\dl<\la<c+\dl<1\}$,
$(s,\la)\rightarrow y(t,x)$ has the expansion
\beq
\label{4-43-2}y(t,x)=s\left(c+(\la-c)\right)+O(s^3+s|\la-c|^2).
\eeq\\
(3) for any $c<-1$, there exist $\ve=\ve(c),\dl=\dl(c)>0$
such that for $(s,\la)\in\{0<s<\ve,c-\dl<\la<c+\dl<-1\}$,
$(s,\la)\rightarrow y(t,x)$ has the expansion
\beq
\label{4-43-4}y(t,x)=-s\left(1+\f{\ln(-c+1)+\ln p}{p}s^p-\f{s^p(\la-c)}{p(c+1)}\right)+O(s^{\min\{p+2,2p+1\}}+s^{p+1}|\la-c|^2).
\eeq\\
(4) for any $-1<c<0$, there exist $\ve=\ve(c),\dl=\dl(c)>0$
such that for $(s,\la)\in\{0<s<\ve,-1<c-\dl<\la<c+\dl<1\}$,
$(s,\la)\rightarrow y(t,x)$ has the expansion
\beq
\label{4-43-5}y(t,x)=s\left(c+(\la-c)\right)+O(s^3+s|\la-c|^2).
\eeq
\end{lemma}
\begin{proof}
We only prove the cases (1), (2) for $c\ge0$ since  the cases (3), (4) can be obtained by the same way.

\bigskip

\noindent(1) If $c>1$, it is similar to the proof of Lemma \ref{lem 4-2} that we adopt the variable transformation
\beq\label{4-44}
\zeta=s^{-p}(1-\mu^{-p})-\ln p,
\eeq
and then $\mu=(1-s^p(\zeta+\ln p))^{-\f{1}{p}}$.

At first, we assume $p\ge 1$. Then \eqref{4-42} becomes
\beq\label{4-45-1}
F_1(s,\la,\zeta)\triangleq \mu(1+e^\zeta)-\mu
e^{\zeta-s^{-p}}+\f{pe^{\zeta}(1-e^{-s^{-p}})}{s}r_0(s\mu)-\la.
\eeq
It is easy to see that for $c> 1$, $F_1(0,c,\ln(c-1))=0$ and
\beq
\p_s\mu=s^{p-1}(\zeta+\ln p)(1-s^p(\zeta+\ln p))^{-\f{1}{p}-1},\
\p_\zeta\mu=\f{s^p}{p}(1-s^p(\zeta+\ln p))^{-\f{1}{p}-1}.\no
\eeq
Then
\ben\label{4-46-1}
\p_s F_1(s,\la,\zeta)&=&\left(1+e^\zeta-e^{\zeta-s^{-p}}\right)\p_s\mu
+ps^{-p-1}\mu e^{\zeta-s^{-p}}\no\\
&&+\f{p e^\zeta}{s^2}\left(e^{-s^{-p}}(1-ps^{-p})-1\right)r_0(s\mu)
+\f{pe^{\zeta}(1-e^{-s^{-p}})}{s}r'_0(s\mu)\left(\mu+s\p_s\mu\right),\no\\
\p_\la F_1(s,\la,\zeta)&=&-1,\no\\
\p_\zeta F_1(s,\la,\zeta)&=&\left(1+e^\zeta-e^{\zeta-s^{-p}}\right)\p_\zeta\mu
+\mu e^\zeta(1-e^{-s^{-p}})+\f{pe^{\zeta}(1-e^{-s^{-p}})}{s}(r_0(s\mu)+r'_0(s\mu)\p_\zeta\mu).\no
\een
This yields
\ben
&&\p_s F_1(0,c,\ln(c-1))=c\left(\ln p+\ln(c-1)\right)\dl_1^p+\f{p(c-1)}{2}r''_0(0),\ \p_\la F_1(0,c,\ln(c-1))=-1,\no\\
&&\p_\zeta F_1(0,c,\ln(c-1))=c-1.
\een
By the implicit function theorem, there exists a unique solution $\zeta(s,\la)$ satisfying $\zeta(0,c)=\ln(c-1)$ and $\zeta(s,\la)$ is $C^\i$ near $(0,c)$ with
\beq\label{4-47-1}
\zeta=\ln(c-1)-\f{c\left(\ln p+\ln(c-1)\right)\dl_1^p+\f{p(c-1)}{2}r''_0(0)}{c-1}s+\f{\la-c}{c-1}+O(s+\la^2).
\eeq
Thus
\ben\label{4-48-1}
\mu&=&(1-s^p(\zeta+\ln p))^{-\f{1}{p}}\no\\
&=&1+\f{\ln(c-1)+\ln p}{p}s^p+\f{s^p(\la-c)}{p(c-1)}+O(s^{p+1}+s^{p}(\la-c)^2).
\een

Secondly, we consider the case of $p\in(0,1)$. Let $\o=s^{p}$
and then $\mu=\left(1-\o(\zeta+\ln p)\right)^{-\f{1}{p}}$. In this case, \eqref{4-45-1} becomes
\beq\label{4-45-2}
F_2(s,\la,\o)\triangleq \mu(1+e^\zeta)-\mu
e^{\zeta-\o^{-1}}+\f{pe^{\zeta}(1-e^{-\o^{-1}})}{\o^{\f{1}{p}}}r_0(\o^\f{1}{p}\mu)-\la.
\eeq
It is obvious that for $c> 1$, $F_2(0,c,\ln(c-1))=0$. By direct computation, we have that
\beq
\p_\o\mu=\f{\zeta+\ln p}{p}(1-\o(\zeta+\ln p))^{-\f{1}{p}-1},\
\p_\zeta\mu=\f{\o}{p}(1-\o(\zeta+\ln p))^{-\f{1}{p}-1}\no
\eeq
and
\ben\label{4-46-2}
\p_\o F_2(\o,\la,\zeta)&=&\left(1+e^\zeta-e^{\zeta-\o^{-1}}\right)\p_\o\mu
+s^{-2}\mu e^{\zeta-\o^{-1}}\no\\
&&+\f{e^\zeta}{\o^{\f{1}{p}+1}}\left(e^{-\o^{-1}}(1-p\o^{-1})-1\right)r_0(\o^\f{1}{p}\mu)
+\f{e^{\zeta}(1-e^{-\o^{-1}})}{\o^\f{1}{p}}r'_0(\o^\f{1}{p}\mu)
\left(\o^{\f{1}{p}-1}\mu+\o^\f{1}{p}\p_\o\mu\right),\no\\
\p_\la F_2(\o,\la,\zeta)&=&-1,\no\\
\p_\zeta F_2(\o,\la,\zeta)&=&\left(1+e^\zeta-e^{\zeta-\o^{-1}}\right)\p_\zeta\mu
+\mu e^\zeta(1-e^{-\o^{-1}})+
\f{pe^{\zeta}(1-e^{-\o^{-1}})}{\o^\f{1}{p}}(r_0(\o^\f{1}{p}\mu)+r'_0(\o^\f{1}{p}\mu)\p_\zeta\mu).\no
\een
This yields
\beq
\p_s F_1(0,c,\ln(c-1))=\f{c(\ln p+\ln(c-1))}{p},\ \p_\la F_1(0,c,\ln(c-1))=-1,\ \p_\zeta F_1(0,c,\ln(c-1))=c-1.
\eeq
By the implicit function theorem, there exists a unique solution $\zeta(s,\la)$ satisfying $\zeta(0,c)=\ln(c-1)$ and $F_2(s,\la,\zeta(s,\la))=0$ with
\beq\label{4-47-2}
\zeta=\ln(c-1)-\f{c\left(\ln p+\ln(c-1)\right)}{p(c-1)}s+\f{\la-c}{c-1}+O(s+\la^2).
\eeq
Thus
\ben\label{4-48-2}
\mu&=&(1-s^p(\zeta+\ln p))^{-\f{1}{p}}\no\\
&=&1+\f{\ln(c-1)+\ln p}{p}s^p+\f{s^p(\la-c)}{p(c-1)}+O(s^{p+1}+s^{p}(\la-c)^2).
\een
Together with \eqref{4-48-1} and \eqref{4-48-2}, this derives \eqref{4-43-1}.

\bigskip

\noindent(2) For $0<c<1$, define

\beq
G(s,\la,\mu)\triangleq\mu(1+\f{1}{p}e^{s^{-p}(1-|\mu|^{-p})})-\f{\mu}{p}e^{-s^{-p}|\mu|^{-p}}
+\f{e^{-|\mu|^{-p}s^{-p}}(e^{s^{-p}}-1)}{s}r_0(s\mu)-\la=0.
\eeq
It is clear that $G(0,c,c)=0$ and
\ben\label{4-49}
\p_s G(s,\la,\mu)&=&s^{-p-1}\left(\mu+\f{pr_0(s\mu)}{s}\right)
\left(e^{s^{-p}(1-\mu^{-p})}(\mu^{-p}-1)-e^{-s^{-p}\mu^{-p}}\mu^{-p}\right)\no\\
&&+\f{s\mu r'_0(s\mu)-r_0(s\mu)}{s^2}
\left(e^{s^{-p}(1-\mu^{-p})}-e^{-s^{-p}\mu^{-p}}\right),\no\\
\p_\la G(s,\la,\mu)&=&-1,\no\\
\p_\mu G(s,\la,\mu)&=&1+s^{-p}\mu^{-p-1}\left(\mu+\f{pr_0(s\mu)}{s}\right)
\left(e^{s^{-p}(1-\mu^{-p})}(1-\mu^{-p})-e^{-s^{-p}\mu^{-p}}\mu^{-p}\right)\no\\
&&+\left(\f{1}{p}+r'_0(s\mu)\right)
\left(e^{s^{-p}(1-\mu^{-p})}-e^{-s^{-p}\mu^{-p}}\right).\no
\een
Then in light of $c<1$, it follows that
\ben\label{4-49-1}
\p_s G(0,c,c)&=&0,\no\\
\p_\la G(0,c,c)&=&-1,\no\\
\p_\mu G(0,c,c)&=&1.\no\een
Thus by the implicit function theorem, there exists a unique solution $\mu=\mu(s,\la)$ satisfying that $\mu(0,c)=c$ and
\beq\label{4-50}
\mu(s,\la)=c+(\la-c)+O_c(s^2+|\la-c|^2).
\eeq
Therefore, by $y=s\mu$, we finish the proof of \eqref{4-43-2}.

%
\end{proof}

\bigskip

\noindent {\it Proof of Theorem \ref{main result2}:}

\noindent (1) It can be obtained by Lemma \ref{lem 4-3}.

\noindent (2) Since we don't get the behavior of $y(t,x)$ for $c=\pm1$ in Lemma \ref{lem 4-6}, we have to choose the domain $\O_{-,t,+}^0$ and $\O_{-,t,-}^0$ as follows
\ben
\label{4-54} \O_{t,-,+}^0&=&\{(t,x):\ 0<s<\ve_0,\ 1-\dl_0<\f{x}{s\tau}<1+\dl_0\},\\
\label{4-55} \O_{t,-,-}^0&=&\{(t,x):\ 0<s<\ve_0,\ -1-\dl_0<\f{x}{s\tau}<-1+\dl_0\},
\een
where $\tau=1-t$, $s=|\ln \tau|^{-\f{1}{p}}$ and $\ve_0$, $\dl_0>0$ sufficiently small.

We only consider the behavior in $\O_{-,t,+}^0$ since the treatment
in $\O_{-,t,-}^0$ is similar. By monotonicity of $y(t,\cdot)$ for each fixed $t\in[0,1]$, we know that
for $(t,x)\in\O_{-,t,+}^0$,
\beq
\label{4-56}
y(t,(1-\dl_0)s\tau)\leq y(t,x)\leq y(t,(1+\dl_0)s\tau).
\eeq

Let's firstly turn to consider $y(t,(1+\dl_0)s\tau)$. Note that $\mu=\f{y(t,(1+\dl_0)s\tau)}{s}$ satisfies
\beq
\label{4-57}
1+\dl_0=\mu\left(1+\f{1}{p}e^{s^{-p}(1-\mu^{-p})}\right)
-\f{\mu}{p}e^{-s^{-p}\mu^{-p}}+\f{e^{s^{-p}}-1}{s}e^{-s^{-p}\mu^{-p}}r_0(s\mu).
\eeq

For $p\ge1$, we let
\beq
\label{4-58}
\zeta=s^{-p}(1-\mu^{-p})-\ln p,
\eeq
and then $\mu=\left(1-s^p(\zeta+\ln p)\right)^{-\f{1}{p}}$.
In this case, \eqref{4-57} becomes
\beq
\label{4-59}
J_1(s,\zeta)\triangleq\mu\left(1+e^\zeta\right)-\mu e^{\zeta-s^{-p}}
+\f{pe^\zeta (1-e^{-s^{-p}})}{s}r_0(s\mu)-(1+\dl_0)=0.
\eeq
Note that if $s\to 0+$, then $\zeta\to(\ln{\dl_0})+$ and $\mu\to1+$.
In addition,
\beq\label{4-60}
\p_s\mu=s^{p-1}(\zeta+\ln p)\left(1-s^p(\zeta+\ln p)\right)^{-\f{1}{p}-1},\
\p_\zeta\mu=\f{s^p}{p}\left(1-s^p(\zeta+\ln p)\right)^{-\f{1}{p}-1},
\eeq
and then by taking $s\to0+$, we have
\beq
\label{4-61}
\p_s\mu(0,\ln\dl_0)=\dl_1^p\left(\ln\dl_0+\ln p\right),\ \p_\zeta\mu(0,\ln\dl_0)=0.
\eeq
On the other hand, it follows from direct computation that
\ben
\p_s J_1(s,\zeta)&=&(1+e^\zeta-e^{\zeta-s^{-p}})\p_s\mu
-ps^{-p-1}\mu e^{\zeta-s^{-p}}\no\\
&&-\f{pe^\zeta\left(e^{-s^{-p}}(ps^{-p}-1)+1\right)}{s^2}r_0(s\mu)
+\f{pe^\zeta(e^{-s^{-p}}-1)}{s}r'_0(s\mu)\left(\mu+s\p_s\mu\right),\no\\
\p_\zeta J_1(s,\zeta)&=&\mu e^\zeta(1-e^{-s^{-p}})+(1+e^\zeta-e^{\zeta-s^{-p}})\p_\zeta\mu
+\f{pe^\zeta(e^{-s^{-p}}-1)}{s}\left(r_0(s\mu)+s\p_\zeta\mu r'_0(s\mu)\right).\no
\een
Together with \eqref{0-8}, this yields
\beq\label{4-62}
\p_s J_1(0,\ln\dl_0)=(1+\dl_0)(\ln\dl_0+\ln p)+\f{p\dl_0}{2}r''_0(0),\ \p_\zeta J_1(0,\ln\dl_0)=1.\eeq
Thus by the implicit function theorem, there exists a function $\zeta(s)$ satisfying $J(s,\zeta(s))=0$ such that
\beq\label{4-64}
\zeta(s)=\ln\dl_0-\left((1+\dl_0)(\ln\dl_0+\ln p)+\f{p\dl_0}{2}r''_0(0)\right)s+O_{\dl_0}(s^2),\quad s\in(0,\ve_0],
\eeq
where $\ve_0=\ve_0(\dl_0)>0$ is small enough. Therefore
\beq\label{4-65}
\mu(s)=\left(1-s^p(\zeta+\ln p)\right)^{-\f{1}{p}}=1+\f{\ln\dl_0+\ln p}{p}s^p+O_{\dl_0}(s^{p+1}),\quad s\in(0,\ve_0].
\eeq

For $p\in(0,1)$, we take $\o=s^p$ and then \eqref{4-59} becomes
\beq\label{4-66}
J_2(\o,\zeta)\triangleq\mu\left(1+e^\zeta\right)-\mu e^{\zeta-\o^{-1}}
+\f{pe^\zeta(1-e^{-\o^{-1}})}{\o^{\f{1}{p}}}r_0(\o^{\f{1}{p}}\mu)-(1+\dl_0)=0,
\eeq
where $\mu=\left(1-\o(\zeta+\ln p)\right)^{-\f{1}{p}}$. By direct computation, one has
\ben
\p_\o J_2(\o,\zeta)&=&(1+e^\zeta-e^{\zeta-\o^{-1}})\p_\o\mu
-\o^{-2}\mu e^{\zeta-\o^{-1}}\no\\
&&-\f{e^\zeta\left(e^{-\o^{-1}}(p\o^{-1}-1)+1\right)}{\o^{\f{1}{p}+1}}r_0(\o^\f{1}{p}\mu)
+\f{pe^\zeta(e^{-\o^{-1}}-1)}{\o^\f{1}{p}}r'_0(\o^\f{1}{p}\mu)
\left(\f{1}{p}\o^{\f{1}{p}-1}\mu+\o\p_\o\mu\right),\no\\
\p_\zeta J_2(\o,\zeta)&=&\mu e^\zeta(1-e^{-\o^{-1}})+(1+e^\zeta-e^{\zeta-\o^{-1}})\p_\zeta\mu
+\f{pe^\zeta(e^{-\o^{-1}}-1)}{\o^\f{1}{p}}\left(r_0(\o^\f{1}{p}\mu)+\o^\f{1}{p}\p_\zeta\mu r'_0(\o^\f{1}{p}\mu)\right),\no
\een
and
\beq\label{4-67}
\p_\o\mu=\f{1}{p}(\zeta+\ln p)\left(1-\o(\zeta+\ln p)\right)^{-\f{1}{p}-1},\
\p_\zeta\mu=\f{1}{p}\th\left(1-\o(\zeta+\ln p)\right)^{-\f{1}{p}-1}.
\eeq
This yields
\beq\label{4-68}
\p_\o\mu(0,\ln\dl_0)=\f{\ln\dl_0+\ln p}{p},\
\p_\zeta\mu(0,\ln\dl_0)=0,
\eeq
and then
\beq\label{4-69}
\p_\o J_2(0,\ln\dl_0)=\f{(1+\dl_0)(\ln\dl_0+\ln p)}{p},\
\p_\zeta J_2(0,\ln\dl_0)=\ln\dl_0.
\eeq
So we can obtain
\beq\label{4-70}
\zeta(\o)=\ln\dl_0-\f{(1+\dl_0)(\ln\dl_0+\ln p)}{p\ln\dl_0}\o+O_{\dl_0}(\o^2),
\eeq
and then
\beq\label{4-71}
\mu=\left(1-\o(\zeta+\ln p)\right)^{-\f{1}{p}}=1+\f{\ln\dl_0+\ln p}{p}s^p+O_{\dl_0}(s^{2p}).
\eeq
By \eqref{4-65} and \eqref{4-71}, we have that for $p>0$ and $0<s<\ve_0$,
\beq\label{4-72}
\mu(s)=1+\f{\ln\dl_0+\ln p}{p}s^p+O_{\dl_0}(s^{\min\{2p,p+1\}}),
\eeq
where $\ve_{0}>0$ is a small constant depending on $\dl_0$. It follows that for $s\in[0,\ve_0)$,
\beq\label{4-73}
y(t,(1+\dl_0)s\tau)=s\mu(s)=s+\f{\ln\dl_0+\ln p}{p}s^{p+1}+O_{\dl_0}(s^{\min\{2p+1,p+2\}}).
\eeq

Secondly, we consider the behavior of $y(t,(1-\dl_0)s\tau)$. For $x=(1-\dl_0)s\tau$, $\mu=\f{y(t,(1-\dl_0)s\tau)}{s}$ satisfies
\beq
\label{4-74}
L(s,\mu)\triangleq\mu\left(1+\f{1}{p}e^{s^{-p}(1-\mu^{-p})}\right)
-\f{\mu}{p}e^{-s^{-p}\mu^{-p}}+\f{e^{s^{-p}}-1}{s}e^{-s^{-p}\mu^{-p}}r_0(s\mu)-(1-\dl_0)=0.
\eeq
It is easy to know $\mu=1-\dl_0$ as $s\to0_+$. Furthermore, by direct computation, one has
\ben
\p_s L(s,\mu)&=& s^{-p-1}\left(\mu+\f{p}{s}r_0(s\mu)\right)
\left(e^{s^{-p}(1-\mu^{-p})}(1-\mu^{-p})-e^{-s^{-p}\mu^{-p}}\mu^{-p}\right)\no\\
&&+\f{e^{s^{-p}(1-\mu^{-p})}-e^{-s^{-p}\mu^{-p}}}{s^2}\left(-r_0(s\mu)+s \mu r'_0(s\mu)\right),\no\\
\p_\mu L(s,\mu)&=&\left(1+\f{1}{p}e^{s^{-p}(1-\mu^{-p})}
-\f{1}{p}e^{-s^{-p}\mu^{-p}}\right)+s^{-p}\mu^{-p}\left(e^{s^{-p}}-1\right)e^{-s^{-p}\mu^{-p}}\no\\
&&+\left(e^{s^{-p}}-1\right)e^{-s^{-p}\mu^{-p}}
\left(ps^{-p-1}\mu^{-p-1}r_0(s\mu)+r'_0(s\mu)\right).\no
\een
By assumption \eqref{0-8}, we have
\beq\label{4-75}
\p_s L(0,1-\dl_0)=0,\quad \p_\mu L(0,1-\dl_0)=1.
\eeq
Then there exists a function
\beq\label{4-76}
\mu(s)=1-\dl_0+O_{\dl_0}(s^2)
\eeq
such that $L(s,\mu(s))=0$ for $s>0$ sufficiently small and dependent on $\dl_0$. It follows that
for $s\in[0,\ve_0)$,
\beq\label{4-77}
y(t,(1-\dl_0)s\tau)=s\mu(s)=(1-\dl_0)s+O_{\dl_0}(s^3).
\eeq
Thus for small fixed $\dl_0>0$, we have that for $s\in[0,\ve_0)$ and $x\in((1-\dl_0)s\tau,(1+\dl_0)s\tau)$,
\beq\label{4-78}
\f{1}{2}s\le y(t,x)\le\f{3}{2}s.
\eeq
Recalling $u(t,x)=u_0(y(t,x))$, we have that
\beq\label{4-79}
|u(t,x)-u(1,0)|\lesssim|y(t,x)|\lesssim s=|\ln(1-t)|^{-\f{1}{p}},
\eeq
and by
\beq\label{4-80}
1+tg'(y(t,x))=(1-t)+\left(\f{1}{p}e^{-|y|^{-p}}+|y|^{-p}e^{-|y|^{-p}}+r'_0(y)\right)
\gtrsim|\ln(1-t)|(1-t),
\eeq
we have that for $(t,x)\in\O_{t,-,+}^0$,
\ben
\label{4-81}\f{\p u}{\p x}(t,x)&=&\f{u'_0(y(t,x))}{1+tg'(y(t,x))}\lesssim|\ln(1-t)|^{-1}(1-t)^{-1},\\
\label{4-82}\f{\p u}{\p t}(t,x)&=&-\f{u'_0(y(t,x))g(y(t,x))}{1+tg'(y(t,x))}\lesssim|\ln(1-t)|^{-1-\f{1}{p}}(1-t)^{-1}.
\een

We decompose the neighbourhood $B$ of $(1,0)$ into $\O_{x,+},\ \O_{x,-},\ \O^{j}_{t,+},\ \O^{j}_{t,-,m}$ for $j=1,2,\ldots,N$ and $\O^{j}_{t,-,+},\ \O^{j}_{t,-,-}$ for $j=0,1,2,\ldots,N$ as follows (see Figure 2 below).

\begin{figure}[h]
\centering
\includegraphics[scale=0.25]{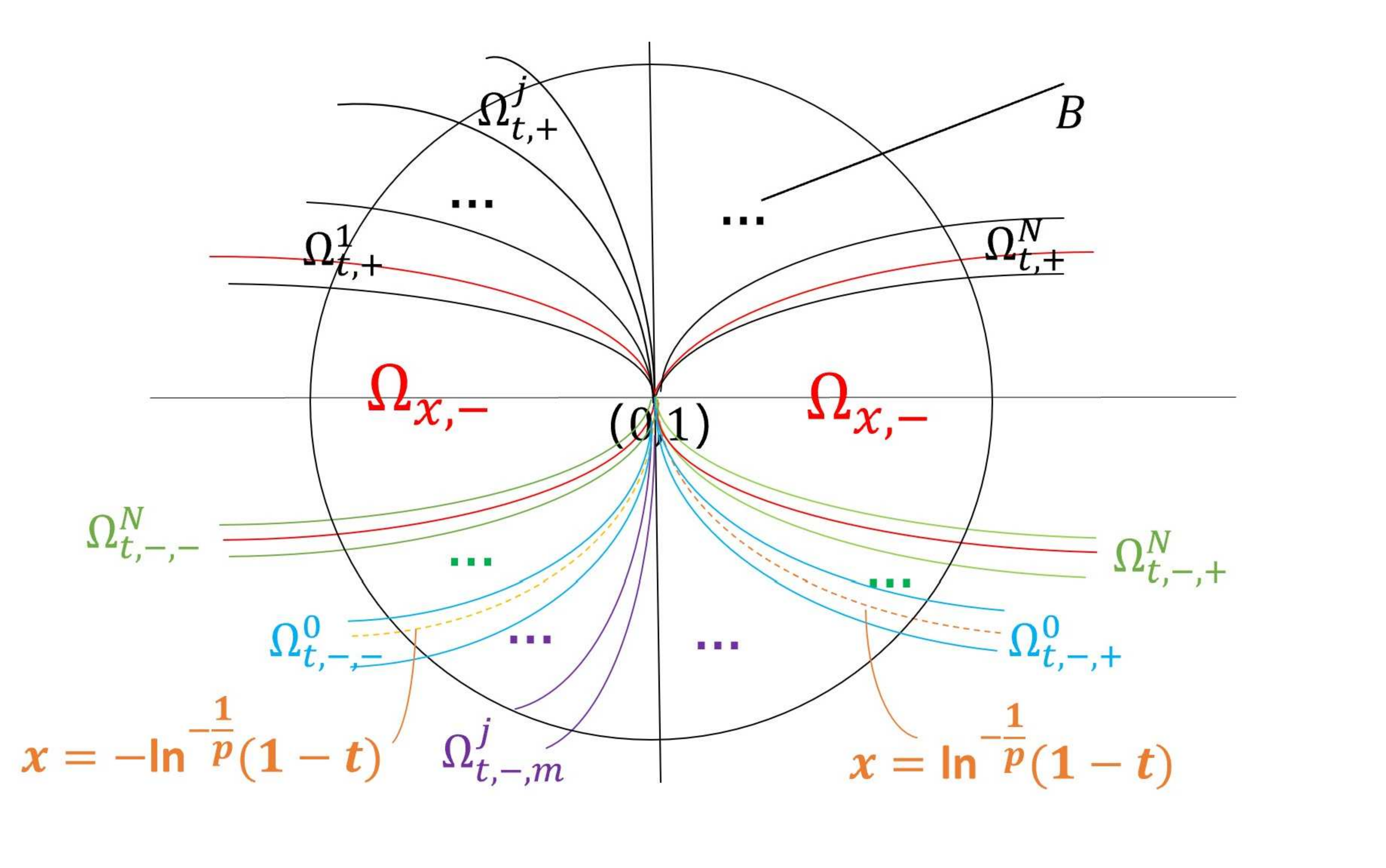}
\caption{\bf Decomposition of $B$}
\end{figure}

\ben
\label{4-83}\O_{x,+}&=&\{(t,x):\ x>0,\ 0<\xi<\dl,\ -\ve<\eta<\ve\},\\
\label{4-84}\O_{x,-}&=&\{(t,x):\ x<0,\ 0<\xi<\dl,\ -\ve<\eta<\ve\},
\een
where $(\xi,\eta)$ are defined in \eqref{4-33}. Taking the suitable constants $c_{1,+}<c_{2,+}<\ldots<c_{N-1,+}<c_{N,+}$, $\{\ve_{j,+}\}_{j=1}^N$ and $\{\dl_{j,+}\}_{j=1}^N$, and setting
\beq\label{4-85}
\O^{j}_{t,+}=\{(t,x):\ 0<s<\ve_{j,+},\ c_{j,+}-\dl_{j,+}<\la<c_{j,+}+\dl_{j,+}\}, s=|\ln(t-1)|^{-\f 1 p},\ \la=\f{x}{s(t-1)},
\eeq
where $B\cap\{t\ge0\}\subset \left(\cup_{j=1}^N\O^{j}_{t,+}\right)\cup\O_{x,+}\cup\O_{x,+}$.
By choosing small $\dl_0>0$, we can define $\O^{0}_{t,-,+}$ and $\O^{0}_{t,-,-}$ as in \eqref{4-54} and \eqref{4-56}.
Meanwhile, taking some suitable constants $\{c_{1,-,m}\}_{j=1}^N$, $\{\ve_{j,-,m}\}_{j=1}^N$ and $\{\dl_{j,-,m}\}_{j=1}^N$ where $m=+,0,-$,
and setting
\beq\label{4-86}
\O^{j}_{t,-,m}=\{(t,x):\ 0<s<\ve_{j,-,m},\ c_{j,-,m}-\dl_{j,-,m}<\la<c_{j,-,m}+\dl_{j,-,m}\}, s=|\ln(1-t)|^{-\f 1 p},\ \la=\f{x}{s(1-t)},
\eeq
where $c_{N,-,-}<\ldots<c_{1,-,-}<-1<c_{1,-,0}<c_{1,-,0}<\ldots<c_{N,-,0}<1<c_{1,-,+}<\ldots<c_{N,-,+}$, such
that $B\cap\{t\le0\}\subset \left(\cup_{j=0}^N\O^{j}_{t,-,+}\right)\cup
\left(\cup_{j=1}^N\O^{j}_{t,-,0}\right)\cup\left(\cup_{j=0}^N\O^{j}_{t,-,-}\right)\cup
\O_{x,+}\cup\O_{x,+}$ holds. Note that in $\O_{x,\pm}$, $y(t,x)\sim\xi$; and in others, $y(t,x)\sim s$.
Then similarly to \eqref{4-79}, we can obtain \eqref{2.1}. On the other hand,
in $\O_{x,\pm}$, $1+tg'(y(t,x))\gtrsim|x||\ln|x||$; and in others, $1+tg'(y(t,x))\gtrsim |t-1||\ln|t-1||$ as in \eqref{4-80}.
Then similarly to the proof for \eqref{4-81} and \eqref{4-82}, we can establish \eqref{2.2} and \eqref{2.3}.
\rightline{$\square$}

\end{document}